\documentclass[a4,12pt]{amsart}       

\usepackage[utf8]{inputenc}
\usepackage[T1]{fontenc}
\usepackage{yfonts}
\usepackage[english]{babel}
  \usepackage[normalem]{ulem}
\usepackage{lipsum}
\usepackage{amsmath}
\usepackage{amsthm}
\usepackage{amssymb}

\usepackage[shortlabels]{enumitem}
\usepackage{graphicx}
\usepackage{mathtools}
\usepackage{hyperref}
\usepackage{amsfonts}
\usepackage{latexsym}
\usepackage{amscd}

\usepackage[all]{xy}
\usepackage[dvipsnames]{xcolor}
\usepackage{stmaryrd}

\usepackage{tikz}
\usetikzlibrary{arrows.meta,positioning}

\DeclareMathOperator{\Ker}{Ker}
\DeclareMathOperator{\End}{End}
\DeclareMathOperator{\Hom}{Hom}

\DeclareMathOperator{\Path}{Path}

\DeclareMathOperator{\Sig}{\Sigma}

\def\Path{\text{Path}}
\def\Mod{\text{-Mod}}

\def\dualita#1#2{\mathrel{
                 \mathop{\vcenter{
                 \offinterlineskip
                 \hbox to 1.2truecm{$\mapsto$}
                 \hbox to 1.2truecm{$\mapsfrom$}}}%
                 }}
\usepackage{aliascnt}
\newtheorem{theorem}{Theorem}[section]
\newaliascnt{lemma}{theorem}
\newtheorem{lemma}[lemma]{Lemma}
\aliascntresetthe{lemma}

\newaliascnt{proposition}{theorem}
\newtheorem{proposition}[proposition]{Proposition}
\aliascntresetthe{proposition}

\newaliascnt{corollary}{theorem}
\newtheorem{corollary}[corollary]{Corollary}
\aliascntresetthe{corollary}

\newaliascnt{definition}{theorem}
\newtheorem{definition}[definition]{Definition}
\aliascntresetthe{definition}

\newaliascnt{remark}{theorem}
\newtheorem{remark}[remark]{Remark}
\aliascntresetthe{remark}
\newaliascnt{example}{theorem}
\newtheorem{example}[example]{Example}
\aliascntresetthe{example}

\usepackage{cleveref}
\crefname{theorem}{theorem}{theorems}
\Crefname{theorem}{Theorem}{Theorems}

\crefname{lemma}{lemma}{lemmas}
\Crefname{lemma}{Lemma}{Lemmas}

\crefname{proposition}{proposition}{propositions}
\Crefname{proposition}{Proposition}{Propositions}

\crefname{corollary}{corollary}{corollaries}
\Crefname{corollary}{Corollary}{Corollaries}

\crefname{definition}{definition}{definitions}
\Crefname{definition}{Definition}{Definitions}

\crefname{remark}{remark}{remarks}
\Crefname{remark}{Remark}{Remarks}

\crefname{example}{example}{examples}
\Crefname{example}{Example}{Examples}

\begin{document}
\title[injective envelopes of simples from exclusive cycles]{The injective envelope of simple modules over Leavitt path algebras II: \\   simples arising from exclusive cycles}
\author{Gene Abrams}
\address{Department of Mathematics, University of Colorado, Colorado
Springs, CO 80918 U.S.A., Orcid https://orcid.org/0000-0001-9046-9811}
\email{abrams@math.uccs.edu}
\author{Francesca Mantese}
\address{Dipartimento di Informatica, Universit\`{a} degli Studi di Verona, I-37134 Verona, Italy, Orcid https://orcid.org/0000-0002-2126-9426}
\email{francesca.mantese@univr.it}
\author{Alberto Tonolo}
\address{Dipartimento di Matematica ``Tullio Levi-Civita'', Universit\`{a} degli Studi di Padova, I-35121, Padova, Italy, Orcid https://orcid.org/0000-0002-9844-3998}
\email{alberto.tonolo@unipd.it}

\begin{abstract}
Let $K$ be any field, $E$ any directed graph, and $L_K(E)$ the associated Leavitt path algebra.    As described first by Chen,  and subsequently generalized by Ara and Rangaswamy, for each cycle $c$ in $E$ one can build the simple left $L_K(E)$-module $V_c^E$, and then more generally  $V_{p(x),c}^E$ (where $p(x)$ is an irreducible polynomial in $K[x,x^{-1}]$).  A cycle $c$ is called {\it exclusive} in case none of the vertices of $c$ is the base of any cycle other than $c$.  In our main result we provide an explicit description of the injective envelope of $V_c^E$, and then more generally of  $V_{p(x),c}^E$,  for each exclusive cycle $c$.   Our method involves defining an $L_K(E)$-module structure on an appropriately-built $K$-vector space of infinite series.  Our main result significantly generalizes previous work of the authors, in that the result holds for all graphs (finite or not), and all exclusive cycles.   
\end{abstract}
\maketitle
\section{Introduction}



The rings known as \emph{Leavitt path algebras}  were introduced two decades ago in \cite{AA05} and \cite{AMP07}.     Over the intervening twenty years, various aspects of  these rings (e.g., their multiplicative and ideal structure; their module-theoretic structure; $\mathbb{Z}$-graded structural aspects; connections between them and other topics including $C^*$-algebras and symbolic dynamics; etc.) have been widely and deeply studied.    
In the current article the three authors continue their investigation into one important aspect of the module-theoretic structure of Leavitt path algebras, to wit, the structure of their injective modules.  The current work  generalizes, corrects,  and provides perspective for  various results appearing in the literature.    

A brief history and overview is in order here.  Let $E$ denote a finite graph,  and $c$ a cycle in  $E$.     Associated to  $c$ is a simple $L_K(E)$-module  $V^E_{c}$, called the {\it Chen} simple module for $c$.  
In \cite{AMT19} the three authors construct a module $M_c$ over $L_K(E)$ (for each cycle $c$ of $E$)  that mimics in many ways the well-known Pr\"ufer abelian groups $\mathbb{Z}_{p^\infty}$, where $p$ is a prime in $\mathbb{Z}$.  Specifically, the  Pr\"ufer $L_K(E)$-module $M_c$ associated to $c$ is uniserial,  artinian, non-noetherian, and has each composition factor isomorphic to $V^E_{c}$.    Whereas the Pr\"ufer abelian groups $\mathbb{Z}_{p^\infty}$ are divisible for all primes $p$, and hence injective over $ \mathbb{Z}$ (thus making   $\mathbb{Z}_{p^\infty}$ the injective envelope of $\mathbb{Z}_{p}$ over $\mathbb{Z}$), the Pr\"ufer $L_K(E)$-module $M_c$ is injective if and only if $c$ is a maximal cycle in $E$ (i.e., there are no cycles $d\neq c$ in $E$ for which there is a path connecting $d$ to $c$).   Consequently, when $c$ is maximal, the module $M_c$ is shown to be the injective envelope of $V^E_{c}$;  see \cite[Theorem 6.4]{AMT19}.
  
 The authors then  proceeded in \cite{AMT21} to identify and describe the injective envelope of the Chen simple $L_K(\mathcal{T})$-modules $V_{p(x),c}^\mathcal{T}$, as well as the injective envelope of the simple left $L_K(\mathcal{T})$-ideal  $L_K(\mathcal{T})w$, where $E = \mathcal{T} := \xymatrix{\bullet \ar@(ul,dl)_c
\ar[r]& \bullet^w}$ and $p(x)$ is an irreducible polynomial in $K[x,x^{-1}]$.   Since $c$ is maximal in $\mathcal T$, the aforementioned \cite[Theorem 6.4]{AMT19} provides the description of the injective envelope of  $V_{p(x),c}^\mathcal{T}$  for $p(x) = x-1$;  the generalization to all  irreducible $p(x)$ follows similarly.  On the other hand, identifying and describing the injective envelope of the simple left ideal $L_K(\mathcal{T})w$ was  a nontrivial task.  In the end, the description of this injective envelope involves the somewhat delicate notion of a ``formal  series in paths ending at $w$'', see \cite[Corollary 29]{AMT21}.  (N.b.: in \cite[Theorem 4.5]{Paper1}, the authors have recently utilized this formal  series idea to obtain an explicit description of the injective envelope of   every simple left ideal of every Leavitt path algebra.)  
  
By a result of Ara and Rangaswamy \cite[Theorem 1.1]{AR14}, the modules $V_{p(x),c}^\mathcal{T}$ ($p(x)$ irreducible in $K[x, x^{-1}]$) and $L_K(\mathcal{T})w$ are in fact up to isomorphism \emph{all} of the simple $L_K(\mathcal{T})$-modules.  Accordingly,   the three authors subsequently endeavored to describe the injective envelopes of all the simple modules over $L_K(E)$, where $E$ is any finite graph included in the hypotheses of   \cite[Theorem 1.1]{AR14}, a so-called \emph{graph with disjoint cycles}.  (In particular, each cycle $c$ in such a graph is necessarily an {\it exclusive} cycle, i.e.,  every vertex in $c$ is the base of only one cycle, $c$ itself.)   We presented our findings in \cite[Theorem 5.2]{AMT24}.   Our method was to try to apply the formal  series idea of \cite{AMT21} (utilized there in the context of the simple left ideal $L_K(\mathcal{T})w$)  to the Pr\"ufer $L_K(E)$-modules $M_c$ arising from cycles, and then use an induction argument on the length of chains of cycles in $E$. 

\smallskip

{\bf An Important Note Regarding an Error in \cite{AMT24}.}  The induction argument  given in \cite{AMT24}  contains an incorrect assertion, which we describe here.  The $K$-vector space $\widehat U_{E,p(c)}$ introduced in \cite[Definition 4.2]{AMT24} is a left $L_K(E)$-module; this module is in general  strictly contained in the injective envelope of the simple $V^E_{p(x),c}$ (contrary to the assertion made in \cite{AMT24}).   More specifically:   in the proof of \cite[Theorem 5.2]{AMT24}, when we wanted to prove that any homomorphism $\chi:I\to\widehat U_{E,c-1}$ extends to a homomorphism $\hat\chi:L_K(E)\to\widehat U_{E,c-1}$ we constructed (see \cite[page 25]{AMT24}) an intermediate homomorphism $\overline\chi:L_K(E)t\to \widehat U_{E,c-1}$ (where $t$ is an appropriate vertex in $E$)  whose image we incorrectly claimed is contained in $\widehat U_{E,c-1}$.  (This image is, however, contained in the corrected version of the injective envelope of $V^E_c$ we construct in the current article).  {\it Thus the statement of \cite[Theorem 5.2]{AMT24} is incorrect as given.} The statement of \cite[Theorem 5.2]{AMT24} becomes correct upon substituting $\widehat U_{E,c-1}$ by the injective envelope of $V^E_c$ as described in the main result of the current article.

\smallskip

However, in the current article, we not only correct our previous invalid description of the injective envelope of $V^E_{p(x),c}$, but also greatly generalize both the results and the framework in which we operate, by using  a completely new approach.
With the previous paragraphs as context, we now describe the  contributions we make in the current work. The ``formal  series''  technique we develop and use not only permits us to achieve  all the now-corrected results of \cite{AMT24}, but also to remove the previously-imposed finiteness restriction on $E$ (to wit, our results here hold for all graphs), and to establish the results for a wider class of cycles (the exclusive cycles) than those discussed in \cite{AMT24}. 

The article is organized as follows.  
 In Section~\ref{section:associativering} we discuss some useful, general ring-theoretic results that are valid for  all rings with local units, as well as present the details undergirding the formal  series idea.  
 In Section~\ref{section:LPA}  we
  present some  results about Leavitt path algebras which are central to the ideas contained in  the sequel.   Importantly,   for a hereditary subset $H$ of $E^0$  we  establish a Morita equivalence between the Leavitt path algebras of the restriction graph $E_H$ and the hedgehog graph ${}_HE$ (Theorem \ref{thm:restricted_hedgehog_Morita}).

In Section \ref{section:injenvcycle}   we then utilize these ideas  to describe the injective envelope of each Chen simple module $V^E_{p(x), c}$, where $c$ is an  exclusive cycle  and  $p(x)$ is irreducible in $K[x,x^{-1}]$.     We achieve the result for 
 $V^E_{x-1,c}$ in Theorem~\ref{injhulltheoremc}.  The central role is played by the description of a left $L_K(E)$-module structure on the appropriate collection of formal  series, see Theorem \ref{prop:formalisTmod}.   (This module structure involves the construction of a Cuntz-Krieger $E$-family acting as endomorphisms on these formal series.)   We then observe how to parlay the ideas used in Theorem~\ref{injhulltheoremc}  to establish the result for all germane $p(x)$.   We again emphasize  that these results posit no restriction on the graph $E$.  
 
 We conclude the article with an Appendix, in which we present some of the computations related to the Cuntz-Krieger $E$-family.

  We believe that the results presented in the current article represent a significant advance in the study of the representations of a Leavitt path algebra. That being said, we readily admit that there is still significant work to be done in this regard.  In particular, we have not yet been able to describe the injective envelope of a Chen simple module of the form $V^E_c$ where $c$ is a non-exclusive cycle.    Specifically, we have not yet been able to describe the injective envelope of the $L_K(R_2)$-module $V^{R_2}_{c}$, where 
  \[ R_2: =\quad\xymatrix{\bullet\ar@(dl,ul)\ar@(ur,dr)}\]  and $c$ is one of the two (non-exclusive) loops in $R_2$.


\section{Some general ring-theoretic results \\ for rings with local units}\label{section:associativering}
Let $S$ be an associative ring (we do not assume that $S$ has a multiplicative identity). 

The ring $S$ is said to have a \emph{set of local units $A$} in case $A$ is a set of idempotents in $S$ having the property that every finite subset of $S$ is contained in a subring of the form $aSa$ for some $a\in A.$ 
The ring $S$ is said to have \emph{enough idempotents} in case there exists a set of nonzero orthogonal idempotents $H$ in $S$ for which the set $\Sig(H)$ of finite sums of distinct elements of $H$ is a set of local units for $S$. 

For a ring $S$ with local units, an abelian group $M$ is a left $S$-{\it module} if there is a (standard) left module action of $S$ on $M$, but with the added proviso that $M$ is $\emph{unitary}$, i.e. that $SM=M$. (This is the appropriate generalization of the requirement that $1_S\cdot m=m$ for all $m$ in a left module $M$ over a unital ring $S$.) In particular, any left ideal of $S$ is a left $S$-submodule of the regular module ${}_SS$. If the condition $SM=M$ is missing, we say that $M$ is a left $S$-{\it premodule}.

For a ring $S$ with local units, the category $S$-Mod of (unitary) left $S$-modules and standard module homomorphisms retains many of the properties of a category of modules over unital rings; importantly for us, that $S$-Mod has enough injectives, and that $S$-Mod admits direct limits.   

The proofs of the following three results can be found in \cite{Paper1}.  

\begin{proposition}\label{prop:Baire}(\cite[Proposition 1.2]{AMT24}). Let $M$ be a left $S$-module and $I$ be a fixed left ideal of $S$. Assume that  any homomorphism $f : X \to M$ from any left
ideal $X\leq I$ extends to a homomorphism $\hat f:I\to M$, and also assume that any homomorphism $g:Y\to M$ from any left ideal $Y\geq I$ extends to a homomorphism $g:S\to M$. Then M is injective.
\end{proposition}

\begin{lemma}\label{lemma:restriction}
Let $M$ be a left $S$-module, $I\leq J$ fixed proper left ideals of $S$. Assume that the restriction map $\Hom_S(J,M)\to\Hom_S(I,M)$ is a monomorphism. If any homomorphism $I\to M$ extends to a homomorphism $S\to M$, then any homomorphism $J\to M$ also extends to a homomorphism $S\to M$.
\end{lemma}

\begin{lemma}\label{lemma:injectivelocalunits}
Suppose $I$ is a two-sided ideal of the ring with local units $S$. Suppose $I$, viewed as a ring,  contains a set of local units. Then any injective left $S/I$-module is an injective left $S$-module.  

Consequently, in this situation, if $E(M)$ is the injective envelope of the simple left $S/I$-module $M$ in $S/I$-Mod, then $E(M)$ is the injective envelope of the simple left $S$-module $M$ in $S$-Mod.  
\end{lemma}

Throughout the article function composition will be written right to left, so that $fg$ means ``first $g$, then $f$''. 


\begin{definition}    \cite[Definition 2.5]{A83}   
Let $S$ be a ring with a set of local units, and let $I$ be an upward directed set.   Suppose that for each $i\in I$ there exists a left $S$-module $X_i$, and for each pair $i\leq j$ in $I$ there exist $S$-homomorphisms
$$\varphi_{ij}:  X_i \to X_j \ \ \mbox{and} \ \ \psi_{ji} : X_j \to X_i.$$
We call the collection $\{ X_i, \varphi_{ij}, \psi_{ji} \ | \ i\in I\}$  {\bf compatible in $S\Mod$} in case

(S.1)  For each $i\in I$, $\varphi_{ii} = \psi_{ii} = id_{X_i}$. 

(S.2)   For each $i\leq j \leq k$ in $I$,    $\varphi_{jk} \varphi_{ij}  =  \varphi_{ik}$  and $ \psi_{ji}\psi_{kj} = \psi_{ki}$.

(S.3)  For each $i\leq j$ in $I$,  $\psi_{ji}\varphi_{ij}  = id_{X_i}$. 

(S.4)   For each $k\geq i,j$ in $I$,   $\varphi_{jk} \psi_{kj} \varphi_{ik} \psi_{ki} = \varphi_{ik}  \psi_{ki}  \varphi_{jk} \psi_{kj}  $. 
\end{definition}

\begin{definition}
Let $H$ be a set of orthogonal idempotents in the ring $S$. The set 
$$\Sig(H)$$ 
consisting of finite sums of distinct elements of $H$ has a standard partial order given by
\[w\leq w'\text{ if } ww'=w'w=w\quad\forall w,w'\in \Sig(H).\]
\end{definition}
\noindent
The set $H$ is a set of enough idempotents for the subring 
\[\Sig(H)S\Sig(H)=\{s \in S \ | \ s = wtw' \ \mbox{for some} \ w,w' \in \Sig(H), t\in S\}\]
of $S$. 

\begin{proposition}\label{prop:newMorita}
Let $S$ be a ring and $H$ a set of orthogonal idempotents in $S$. Then $\Sig(H) S \Sig(H) \subseteq S \Sig(H) S$.   If $S\Sig(H)S$ has local units, then $S\Sig(H)S$ and $\Sig(H)S\Sig(H)$ are Morita equivalent.   
\end{proposition}
\begin{proof}
The first containment follows immediately as $(\Sig(H))^2 = \Sig(H)$.  We note that easily $(S \Sig(H) S)^2 =S \Sig(H) S$; in particular, this gives that  for each $w\in \Sig(H)$,   $X_w := S\Sig(H)Sw$ is a   left $S\Sig(H)S$-module.  We impose the ordering from $\Sig(H)$ to the set $\{X_w \ | \ w\in \Sig(H)\}$.      For $w,w',w''$ having $w'' = w+  w'$ in $\Sig(H)$ with $w \perp w'$   we have both  $X_w \leq X_{w''}$ and $X_{w'} \leq X_{w''}$.     Let $\varphi_{w,w''} : X_w \to X_{w''} $ denote the standard injection and let $\psi_{w'',w} : X_{w''} \to X_w$ denote the standard projection along $X_{w'}$.   

 A tedious but obvious computation yields that the collection \[\{X_w, \varphi_{w,w'}, \psi_{w', w} \ | \ w\in \Sig(H)\}\] is compatible in $S\Sig(H)S$-Mod.
 
Now consider the left $S\Sig(H)S$-module $P := \underrightarrow{{\rm lim}}_{w\in \Sig(H)}(X_w, \varphi_{w,w'})$.   Define ${}_{S\Sig(H)S}Q:= \bigoplus_{s\in S\Sig(H)S} P$;  so $Q$ is the external direct sum of copies of $P$, indexed by the set $S\Sig(H)S$.   Define
$$\xi:  Q \to S\Sig(H)S \ \ \ \mbox{via} \ \ \  (p_s) \mapsto \sum_{s\in S\Sig(H)S} p_s s .$$
Easily $\xi$ is a left $S\Sig(H)S$-module homomorphism.   Let $s\in S\Sig(H)S$; write $s = \sum_{i=1}^n x_i w_i y_i$ with $x_i, y_i \in S, w_i \in \Sig(H)$  ($1\leq i \leq n$, some $n\in \mathbb{N}$).   Let $q\in Q$ be the element with $x_i w_i$ in the $y_i$-component (for $1\leq i \leq n$), and $0$ elsewhere.  Then $(q)\xi = \sum_{s\in S} q_s\cdot s = \sum_{i=1}^n  x_iw_i \cdot y_i = s$. Thus, $\xi$ is surjective and so $P$ is a generator for $S\Sig(H)S$-Mod.   
    By standard ring theory, each $w(S\Sig(H)S)w$ is isomorphic to ${\rm End}_{S\Sig(H)S}(S\Sig(H)Sw)$.  Moreover, the diagram  
\[   
\xymatrix{wS\Sig(H)Sw \ar[r] \ar[d]^{inc} & {\rm End}_{S\Sig(H)S}(S\Sig(H)Sw)\ar[d]^{f \mapsto \psi_{w',w}f\varphi_{w,w'}}\\
w'S\Sig(H)Sw' \ar[r]& {\rm End}_{S\Sig(H)S}(S\Sig(H)Sw') }
 \]
of ring homomorphisms commutes. 
Now clearly $$\Sig(H)S\Sig(H) = \bigcup_{w\in \Sig(H)} w S\Sig(H)S w.$$
Thus the desired Morita equivalence between the module categories $S\Sig(H)S\Mod$ and $\Sig(H)S\Sig(H)\Mod$   follows directly from  \cite[Theorem 4.2]{A83}.
\end{proof}

\begin{remark}
 In the previous proof we assume that $S\Sig(H)S$ has local units in order that we may apply \cite[Theorem 4.2]{A83}. 
The assumption that  $S\Sig(H)S$ has local units is not superfluous,  as demonstrated by the following example. Let $S$ be the ring consisting of those $\mathbb{N} \times \mathbb{N}$  matrices, having  coefficients in a field $K$, that differ from an infinite scalar matrix only in finitely many columns. Then $S$ is a ring, indeed a ring with unit. The matrix $E_{11}$ with 1 in position $(1,1)$ and 0 elsewhere is an idempotent in $S$; and $\Sigma(E_{11}) = \{E_{11}\}$.  
The two-sided ideal $T = SE_{11}S = S \Sigma(E_{11}) S$ of $S$  consists of those matrices having nonzero entries in only finitely many columns; $T$ is easily shown to be  a ring that does not contain a set of local units.
\end{remark}

\begin{corollary}\label{cor:newMorita}
Denote by $I:=I(H)$ the two-sided ideal generated by a set of orthogonal idempotents $H$ in $S$. Assume that $I$ has local units. Then, as rings, $I$ and $\Sig(H)I\Sig(H)$ are Morita equivalent.
\end{corollary}
\begin{proof}We have
\[I=S\Sig(H) S=(S\Sig(H))\Sig(H)(\Sig(H) S)\subseteq I\Sig(H) I\subseteq I\cdot I\subseteq I\quad \text{and}\]
\[
\Sig(H) I\Sig(H)\subseteq \Sig(H) S\Sig(H) =\Sig(H)(\Sig(H) S)\Sig(H)\subseteq \Sig(H) I\Sig(H) \]
and hence all the above displayed containments are indeed equalities.
Therefore the result follows from  \Cref{prop:newMorita}.
\end{proof}


We conclude this section by making some observations about the notion of a vector space of formal series.  

Let $K$ be a field. For each set $Y$, we denote by $K^Y$ the $K$-vector space of all the functions $Y\to K$. Any element $\mathfrak p$ in $K^Y$ can be represented by a \emph{formal series in $Y$}, i.e., an infinite sum
\[\sum_{y\in Y}k_yy.\]
The latter can be identified with the function $Y\to K$ that maps $\hat y$ to $k_{\hat y}$ for each $\hat y\in Y$:
\[
\left(\sum_{y\in Y}k_yy\right)(\hat y)=k_{\hat y}\quad \forall \hat y\in Y.
\]
The set of formal series in $Y$ is denoted by $K[[Y]]$.
The role of variables $y\in Y$ is that of a ``placeholder" to identify the images of the elements of $Y$. The vector space structure of $K^Y$ is reflected in the usual operations of sum and scalar product  for the elements of $K[[Y]]$.
Having fixed any set $X$ and any injective map $f:Y\to X$, 
we can denote also by
\[
\sum_{y\in Y}k_yf(y)
\]
the function $Y\to K$ that maps $\hat y$ to $k_{\hat y}$ for each $\hat y\in Y$, where we use, instead of $y$, the expression $f(y)$ as placeholder.  More formally,
\[
\left(\sum_{y\in Y}k_yf(y)\right)(\hat y)=k_{\hat y}\quad \forall \hat y\in Y.
\]
This seemingly-overly-cumbersome notation will serve us well in the sequel.

\section{Some general results about \\ Leavitt path algebras  over arbitrary graphs}\label{section:LPA}

We assume the reader is familiar with the notion of a {\it Leavitt path algebra};  see \cite{AAM}  for additional details and background information.  We start this section by setting some notation.  

We let $E=(E^0, E^1, s,r)$ denote an arbitrary directed graph, $K$ any field,  and $L_K(E)$  the associated Leavitt path $K$-algebra.  $L_K(E)$ is the quotient of the path $K$-algebra $K\widehat E$ on the \emph{extended graph} $\widehat E$ by the ideal generated by the Cuntz-Krieger relations (see \cite[Definition 1.2.3]{AAM}). The algebra $L_K(E)$ always has enough idempotents: the vertices in $E^0$ are nonzero orthogonal idempotents, and the set of finite sums of distinct elements of $E^0$ constitutes a set of local units for $L_K(E)$. $L_K(E)$ has multiplicative identity if and only if the set $E^0$ of vertices is finite, in which case $1_{L_K(E)} = \sum_{v\in E^0}v.$

We denote by $\Path(E)$ the set of all paths in $E$  (i.e., the finite sequences of edges $e_1\cdots e_m$ such that $r(e_i)=s(e_{i+1})$, $1\leq i<m$), together with the vertex set $E^0$.    A {\it cycle} is a path $e_1 e_2\cdots e_m$ in $E$ for which $r(e_m)=s(e_1)$, and the vertices $s(e_1), s(e_2), \dots , s(e_{m})$ are distinct.  
We define a preorder $\geq$ on $E^0$ given by:
\[v\geq w\ \ \text{ in case there is }\mu\in\text{Path}(E) \ \mbox{for which} \  s(\mu)=v, r(\mu)=w.\]
If $H_0\subseteq E^0$ then \emph{the tree $T(H_0)$} is the set
\[T(H_0):=\{w\in E^0\mid \exists v\in H_0\text{ such that }v\geq w\}.\]
A vertex $v$ in $E^0$ is 
\begin{itemize}
    \item \emph{regular} if the set $s^{-1}(v)$ of edges with source $v$ has a finite cardinality greater than or equal to 1. 
\item  an \emph{infinite emitter} if $s^{-1}(v)$ is infinite. 
\item a \emph{sink} if $s^{-1}(v)=\emptyset$.
\end{itemize}
We denote the set of regular vertices (resp., infinite emitters) of $E$  by ${\rm Reg}(E)$ (resp., ${\rm Inf}(E)$).  
A set of vertices $H\subseteq E^0$ is 
\begin{itemize}
    \item \emph{hereditary} if whenever $v\in H$ and $w\in E^0$ for which $v\geq w$, then $w\in H$.
    \item \emph{saturated} if whenever a regular vertex $v$ has the property that $\{r(e)\mid e\in E^1, s(e)=v\}\subseteq H$, then $v\in H$.
\end{itemize}
It is easy to check that $T(H_0)$ is the hereditary closure of $H_0$.


To any nonempty hereditary subset of $E^0$  we can associate  two new graphs and their corresponding Leavitt path algebras.

\begin{definition}(\cite[Definitions~2.2.21, ~2.5.16]{AAM})\label{def:hedgehog}
Let $H$ be a nonempty hereditary subset of $E^0$. We denote by $F_E(H)$ the set
\[
\begin{aligned}
F_E(H):=\{\alpha\in & \text{Path}(E)\mid \alpha=e_1\cdots e_n,\text{ with }s(e_1)\in E^0\setminus H,\\
&r(e_i)\in E^0\setminus H\ \forall 1\leq i<n, r(e_n)\in H\}
\end{aligned}
\]
\begin{itemize}
\item We denote by $E_H=(E_H^0,E_H^1, s_H, r_H)$ the \emph{restriction graph}:
\[ E_H^0=H,\quad E_H^1=\{e\in E^1\mid s(e)\in H\}, \quad (s_H,r_H)=(s_{\mid E_H^1},r_{\mid E_H^1}).\]
\item We denote by ${}_HE=({}_HE^0,{}_HE^1,s',r')$ the \emph{hedgehog graph}:
\[{}_HE^0=H\cup F_E(H)\quad\text{and}\quad {}_HE^1=\{e\in E^1\mid s(e)\in H\}\cup F'_E(H)\]
where $F'_E(H):=\{\alpha':\alpha\in F_E(H)\}$ denotes another copy of $F_E(H)$.
The source and range functions $s'$ and $r'$ are defined by setting $s'(e)=s(e)$ and $r'(e)=r(e)$ for every $e\in E^1$ such that $s(e)\in H$; and by setting $s'(\alpha')=\alpha$ and $r'(\alpha')=r(\alpha)$ for all $\alpha'\in F'_E(H)$.
\end{itemize}
\end{definition}

The notion of \emph{restricted graph} $E_H$ is self-explanatory, while the notion of \emph{hedgehog graph} ${}_HE$ deserves further discussion. Intuitively,  the vertex set ${}_HE^0$ can be viewed as $H$, together with a new vertex corresponding to each path in $E$ of length $\geq 1$ that ends at a vertex in $H$, but for which none of the previous edges in the path ends at a vertex in $H$.  The edge set ${}_HE^1$ can be viewed as $E_H^1$, together with a new edge going to $H$ for each new vertex that we added to $H$ in ${}_HE^0$. Consequently, in ${}_HE$ the only paths entering subgraph $E_H$ have common length 1. 
\begin{lemma}\label{lem:restrictedversushedgehog}  Let $H$ be a nonempty hereditary subset of $E^0$.
\begin{enumerate}
    \item  The Leavitt path algebra $L_K(E_H)$ associated with the restricted graph  $E_H$ coincides with the subring $\Sig(H)L_K(E)\Sig(H)$ of $L_K(E)$.  

\item   The two-sided ideal $I(H)$ of $L_K(E)$ generated by $H$,  viewed as a $K$-algebra, is isomorphic to $L_K({}_HE)$. In particular, $I(H)$ has local units. 
\end{enumerate}
\end{lemma}

\begin{proof}
Statement $(1)$ is easy, while Statement $(2)$ appears as \cite[Theorem~2.5.19]{AAM}\label{lemma:Ilocunits}. 
\end{proof}

The following result provides a strong and useful connection between the Leavitt path algebras of the two graphs $E_H$ and ${}_HE$.
Specifically, combining Lemma \ref{lem:restrictedversushedgehog} with Corollary \ref{cor:newMorita}, we have established
\begin{theorem}\label{thm:restricted_hedgehog_Morita}
Let $E$ be any graph and let $H$ be any nonempty hereditary subset of $E^0$.  Then the Leavitt path algebras $L_K(E_H)$ and $L_K({}_HE)$ are Morita equivalent.
\end{theorem}

It will be important for us to investigate properties of the set $F_E(H)$ introduced in Definition \ref{def:hedgehog} (which by definition does not include $H$ itself), expanded to include $H$; we formalize that idea  here.  
\begin{definition}\label{def:estesa}
Let $H_0$ be a nonempty subset of $E^0$ and $H=T(H_0)$ its hereditary closure. We denote by $F^+_E(H_0)$ the set 
\[F^+_E(H_0) \ := \ H_0 \ \sqcup \ \{\alpha\in F_E(H):r(\alpha)\in H_0\}.\] 
\end{definition}



\begin{lemma}\label{lemma:mu*lambda}
Let $H_0$ be a nonempty subset of $E^0$.  
Then for each $\mu, \lambda\in F^+_E(H_0)$ we have
\[\mu^*\lambda=\begin{cases}
    r(\mu)  & \text{if }\mu=\lambda, \\
      0& \text{otherwise}.
\end{cases}\]
\end{lemma}
\begin{proof}
By \cite[Lemma~1.2.12]{AAM} $\mu^*\lambda\not=0$ implies $\mu=\lambda\kappa$ for some $\kappa\in \text{Path}(E)$ or $\lambda=\mu\sigma$ for some $\sigma\in \text{Path}(E)$. If $\mu=\lambda\kappa$ and $\kappa\not=r(\lambda)$, then the ranges of two edges in $\mu$ would be in $H_0\subseteq T(H_0)$, contrary to the property of the paths in $F^+_E(H_0)$ (see \Cref{def:hedgehog,def:estesa}). Analogously if $\lambda=\mu\sigma$ and $\sigma\not=r(\mu)$.
\end{proof}

The proof of the following result appears in \cite{Paper1}. 
\begin{proposition}
\label{lemma:Rmu*}
The two-sided ideal $I(H)$ of $L_K(E)$ generated by a nonempty hereditary subset $H$ of $E^0$ is equal, as a left $L_K(E)$-ideal, to
\[ \sum_{\mu\in F^+_E(H)}L_K(E)\mu^* \ = \ \bigoplus_{\mu\in F^+_E(H)}L_K(E)\mu^*,\]
which in turn is isomorphic as left $L_K(E)$-modules to the external direct sum
\[ \bigoplus_{\mu\in F^+_E(H)}L_K(E)r(\mu).\]
\end{proposition}

The remainder of this section is taken up in presenting the ideas and properties attendant to the simple $L_K(E)$-modules associated to cycles in $E$, and their corresponding Pr\"ufer modules.

Let $c=e_1\cdots e_n$ be a cycle in $E$. We denote by $v_i$ the source $s(e_i)$ of the edge $e_i$, $1\leq i\leq n$ and set $v:=v_1$ the source of the cycle.
Denote by 
\[p_{v_i,v_j}\]
the shortest real path along the cycle with source $v_i=s(e_i)$ and range $v_j=s(e_j)$. E.g., if $c=e_1e_2e_3$, we have
\[p_{v_2,v_1}=e_2e_3,\quad p_{v_2,v_2}=v_2,\quad p_{v_3,v_2}=e_3e_1.
\]

\begin{lemma}\label{computewithpsubwcycle}
If $w,w',w''\in c^0$, then in $L_K(E)$ we have
\begin{enumerate}
\item $p^*_{w,w'} p_{w,w''}=\begin{cases}
    p^*_{w'',w'} & \text{if }|p_{w,w'}|\geq |p_{w,w''}|, \\
       p_{w',w''} & \text{otherwise}.
\end{cases}$
\item $p_{w,w'}^* p_{v,w}^* = \begin{cases}
     p_{v,w'}^* & \text{if }|p_{v,w}|+|p_{w,w'}|<|c|, \\
     p_{v,w'}^* c^*& \text{otherwise.}
\end{cases}$
\end{enumerate}
\end{lemma}
\begin{proof}
(1) For any edge $e_i$ of $c$ we have $e_i^*e_i=r(e_i)$: the formula follows easily.

(2) The composition $p_{w,w'}^* p_{v,w}^*$ has source $w'$ and range $v$. It ends with a traverse of $c^*$ if and only if $|p_{v,w}|+|p_{w,w'}|\geq |c|$.
\end{proof}


Given an infinite path $q:=q_1q_2\cdots q_n\cdots$ in $E$ (see e.g. \cite[Definition 2.9.4]{AAM} for complete description) and an integer $n\geq 1$, we denote by
\[\tau_{\leq n}q:=q_1q_2\cdots q_n,\quad \tau_{> n}q:=q_{n+1}q_{n+2}\cdots
\]
the $n$-prefix and the $n$-tail of $q$.
To the cycle $c=e_1\cdots e_n$ in $E$, one can associate the 
$$ \mbox{Chen simple left } L_K(E) \mbox{-module} \  V^E_{c}.$$


As first described in \cite{Ch12}, 
as a $K$-vector space, $V^E_{c}$ has as a basis the infinite paths in $E$ \emph{tail equivalent} to $c^\infty$, i.e., the infinite paths $q$ for which there exist $m\in\mathbb N$ such that $\tau_{>m}q=c^\infty:=ccc\cdots$.
Then $V^E_{c}$ is made into a left $L_K(E)$-module by defining, for all infinite paths $q$ tail equivalent to $c^\infty$,  $u\in E^0$, $f\in E^1$,
\begin{itemize}
\item $u\cdot q=\begin{cases}
 q  & \text{if }u=s(q), \\
    0  & \text{otherwise}.
\end{cases}$
\item $f\cdot q=\begin{cases}
 fq   & \text{if }r(f)=s(q), \\
    0  & \text{otherwise}.
    \end{cases}$\item $f^*\cdot q=\begin{cases}
 \tau_{>1}q=q_2q_3\cdots   & \text{if }q=fq_2q_3\cdots, \\
    0  & \text{otherwise}.
    \end{cases}$
\end{itemize}
In particular, for each vertex $u\not=s(c)=:v$ we have  $u\cdot c^\infty=0$, and $(c-v)c^\infty=0$. The kernel of the map
$L_K(E)v\to V^E_{c}$, $r\mapsto rc^\infty$ is the left ideal $L_K(E)(c-v)$. 

Following \cite{Ch12} and \cite{AR14}, it is possible to get in two steps  broader classes of simple $L_K(E)$-modules.

\underline{First step}. For each $0\neq a\in K$, denote by $\sigma_{c,a}$ the gauge automorphism of $L_K(E)$ associated with $c = e_1e_2\cdots e_n$ and $a$. Then $\sigma_{c,a}$ maps each vertex $u \in E^0$ to itself, each edge $f \in E^1 \setminus {e_1}$ to itself, and similarly each $f^* \in (E^1)^* \setminus {e_1^*}$ to itself. However, it sends $e_1$ to $a e_1$ and $e_1^*$ to $a^{-1} e_1^*$.
In the special case where $a = 1$, the automorphism $\sigma_{c,1}$ clearly reduces to the identity map on $L_K(E)$.

For $M \in L_K(E)\text{-Mod}$, we define the twisted module $M^{\sigma_{c,a}}$ by taking $M^{\sigma_{c,a}} = M$ as an abelian group, but modifying the left $L_K(E)$-action via
$$\ell\star m:=\sigma_{c,a}(\ell)m$$
for all $\ell \in L_K(E)$ and $m \in M$.

Consider the set $\Sig(E^0)$ of finite sums of distinct elements of $E^0$ endowed with the partial order
\[w\leq w' \text{ if }ww'=w'w=w\qquad w,w'\in \Sig(E^0).\]
By \cite[Lemma 1.5]{A83} for each left $L_K(E)$-module $M$ we have
\[
\varinjlim_{w\in\Sig(E^0)} L_K(E)w\otimes_{wL_K(E)w} wM\cong M.
\]
The automorphism $\sigma_{c,a}$ determines an auto-equivalence of the category $L_K(E)\text{-Mod}$ given by the functor
\[
 \varinjlim_{w\in\Sig(E^0)} L_K(E)^{\sigma_{c,a}}w\otimes_{wL_K(E)w} w\cdot-:L_K(E)\text{-Mod}\to L_K(E)\text{-Mod}
\]
 which assigns to each $L_K(E)$-module $M$ the twisted module 
 $$M^{\sigma_{c,a}} \cong  \varinjlim_{w\in\Sig(E^0)} L_K(E)^{\sigma_{c,a}}w\otimes_{wL_K(E)w} wM.$$
In particular, for every $a \in K$ with $a \neq 0$, the twisted module $(V^E_c)^{\sigma_{c,a}}$ is a simple left $L_K(E)$-module.

\underline{Second step}.
Let $K$ be any field, and let $p(x) \in K[x,x^{-1}]$ be an irreducible polynomial. It is not restrictive to assume that $p(x)$ is a \emph{basic irreducible} polynomial in $K[x,x^{-1}]$
—meaning that $p(x)$ is irreducible in $K[x]$ and $p(0)=-1$. Assume $p(x)=p_\ell x^\ell+\cdots+p_1x-1$. 

$$\mbox{If }a_p(x):=p_\ell x^{\ell-1}+\cdots+p_1, \mbox{ then } p(x)=xa_p(x)-1.$$ Define $K' := K[x,x^{-1}]/\langle p(x) \rangle$, and let $\overline{x}$ denote the coset $x + \langle p(x) \rangle \in K'$. Clearly $\overline{x}\not=0$ and hence invertible in $K'$: 
$\overline{x}^{-1}=a_p(\overline x)$.
Using the gauge automorphism $\sigma_{c,\overline x}$ of $L_{K'}(E)$ we can consider the simple left $L_{K'}(E)$-module $(V_c^E)^{\sigma_{c,\overline x}}$.


\begin{definition}
We define $(V_c^{E})^{p}$ to be the left $L_K(E)$-module obtained by restricting scalars from $K'$ to $K$ on the twisted left $L_{K'}(E)$-module $(V_c^{E})^{\sigma_{c,\overline{x}}}$.
\end{definition}
The left $L_K(E)$-module $(V_c^{E})^{p}$ is simple (see \cite[Lemma~3.3]{AR14} or \cite[Proposition 4.2]{MT26}).

For $p(x)=p_\ell x^\ell+\cdots+p_1x-1$ a basic irreducible element of $K[x,x^{-1}]$, and $c$ a cycle in $E$ with $s(c)=v$, we define
$$p(c) := p_\ell c^\ell+\cdots+p_1c- v.$$
By \cite[Theorem 4.3]{MT26}, the simple left $L_K(E)$-module $(V_c^{E})^{p}$ is isomorphic to $L_K(E)v/L_K(E)p(c)$. Observe that if $a(x)=a\in K$, then $p(x)=ax-1$, and $(V_c^{E})^{p} =       (V^E_c)^{\sigma_{c,a}}        $; in particular if $a(x)=1_K$, then $p(x)=x-1$, and $(V_c^{E})^{p}=V^E_{c}$.

\begin{definition}\label{Pruferdef}
Let $p(x)$ be a basic irreducible polynomial in $K[x,x^{-1}]$  and let $c$ be any cycle. The direct limit 
$$U^E_{p(x),c}:=\varinjlim_i L_K(E)v/L_K(E)p(c)^i$$
is called the \emph{Pr\"ufer left $L_K(E)$-module associated to $c$ and $p(x)$}.
We denote by 
$$\alpha^E_{p(x),c,i}:= \psi_i(v+L_K(E)p(c)^i),$$ where
$\{\psi_i:L_K(E)v/L_K(E)p(c)^i\to U^E_{p(x),c} \ | \  i\geq 1\}$ 
are the canonical embeddings in the direct limit.
\end{definition}
In particular,  left annihilator of $\alpha^E_{p(x),c,i}$ in $L_K(E)v$ is $L_K(E)p(c)^i$.   


\smallskip
The following properties play a fundamental role in \Cref{section:injenvcycle}.

\begin{lemma}\label{rem:contialpha-pversion}\label{rem:contialpha}\label{lemma:nuovo c1c1*}
We define $\alpha^E_{p(x),c,j}=0$ for all $j\leq 0.$  Then, for each $i\geq 1$,
\begin{enumerate}
\item $p(c)\alpha^E_{p(x),c,i}=\alpha^E_{p(x),c,i-1}\quad\forall i\geq 1$.
\item $u\alpha^E_{p(x),c,i}=\begin{cases}
      \alpha^E_{p(x),c,i}& \text{if }u=s(c), \\
     0 & \text{if }u\in E^0\setminus\{s(c)\}.
\end{cases}$
\item $ca_p(c)\alpha^E_{p(x),c,i}=\alpha^E_{p(x),c,i}+\alpha^E_{p(x),c,i-1} $
\item $c^* \alpha^E_{p(x),c,i}= a_p(c)(\sum_{j=0}^{i-1}(-1)^j \alpha^E_{p(x),c,i-j})$.
\item $f^*p_{s(c),s(f)}^*\alpha^E_{p(x),c,i}=0$ for each exit $f$ of $c$ and $i\geq 1$;
\item $e_n a_p(c)\alpha^E_{p(x),c,i}=p^*_{v,v_n}(\alpha^E_{p(x),c,i}+\alpha^E_{p(x),c,i-1})$ for all $i\geq 1$;
\item $e_{n-j}p^*_{v,v_{n-j+1}}\alpha^E_{p(x),c,i}=p^*_{v,v_{n-j}}\alpha^E_{p(x),c,i}$ for each $1\leq j\leq n-1$. In particular $e_1e^*_1 \alpha^E_{p(x),c,i}=\alpha^E_{p(x),c,i}$.
\end{enumerate}
\end{lemma}

\begin{proof}
The equality (1) follows from 
\begin{align*}
p(c)\alpha^E_{p(x),c,i}&=p(c)\psi_i(v+L_K(E)p(c)^i)\\
&=\psi_i(p(c)+L_K(E)p(c)^i)\\
&=\psi_{i-1}(v+L_K(E)p(c)^{i-1})=\alpha^E_{p(x),c,i-1}.
\end{align*}

The equality (2) follows from 
\begin{align*}u\alpha^E_{p(x),c,i}&=u\psi_i(v+L_K(E)p(c)^i)\\
&=\psi_i(u\cdot v+L_K(E)p(c)^i)\\
&=\begin{cases}
  \psi_i(v+L_K(E)p(c)^i) =\alpha^E_{p(x),c,i}   & \text{if }u=v, \\
 \psi_i(0+L_K(E)p(c)^i)=0     & \text{otherwise}.
\end{cases}
\end{align*}

(3) Since $p(c)\alpha^E_{p(x),c,i}=(ca_p(c)-v)\alpha^E_{p(x),c,i}$, we get
\[
ca_p(c)\alpha^E_{p(x),c,i}=\alpha^E_{p(x),c,i}+\alpha^E_{p(x),c,i-1}.
\]

(4) By property (3) we get
\begin{align*}
    c^*\alpha^E_{p(x),c,i}&=c^*\big(ca_p(c)\alpha^E_{p(x),c,i}-\alpha^E_{p(x),c,i-1}\big)\\
    &=a_p(c)\alpha^E_{p(x),c,i}-c^*\alpha^E_{p(x),c,i-1}.
\end{align*}
Using this, (4) follows by an induction argument.

(5) Since   $p^*_{s(c),s(f)}c$ is a subpath of $c$ of length $\geq 1$, and   since $f^*e_i=0$ for each edge $e_i$ of $c$,  we get $f^*p^*_{s(c),s(f)}c = 0.$  With the previously noted $p(c)\alpha^E_{p(x),c,i}=(ca_p(c)-v)\alpha^E_{p(x),c,i}$, this gives
\begin{align*}
f^*p^*_{s(c),s(f)}\alpha^E_{p(x),c,i}&=-f^*p^*_{s(c),s(f)}p(c)\alpha^E_{p(x),c,i}\\
&=-f^*p^*_{s(c),s(f)}\alpha^E_{p(x),c,i-1} \ \ \ \ \mbox{by (1)}\\
&= \ \cdots \ =(-1)^{i-1}f^*p^*_{s(c),s(f)}\alpha^E_{p(x),c,1}\\
&=(-1)^{i}f^*p^*_{s(c),s(f)}p(c)\alpha^E_{p(x),c,1}=0.
\end{align*}

(6) We have
\begin{align*}
    e_n a_p(c)\alpha^E_{p(x),c,i}&=p^*_{v,v_n}ca_p(c)\alpha^E_{p(x),c,i}\\
    &=p^*_{v,v_n}((\alpha^E_{p(x),c,i}+\alpha^E_{p(x),c,i-1}).
\end{align*}

(7) We proceed by induction on $i\geq 1$, and $1\leq j\leq n-1$. Let $i=1=j$:
\begin{align*}
    e_{n-1}p^*_{v,v_n}\alpha^E_{p(x),c,1}&=e_{n-1}p^*_{v,v_n}ca_p(c)\alpha^E_{p(x),c,1}\\
    &=e_{n-1}e_n a_p(c)\alpha^E_{p(x),c,1}\\
    &=p^*_{v,v_{n-1}}ca_p(c)\alpha^E_{p(x),c,1}\\
    &=p^*_{v,v_{n-1}}\alpha^E_{p(x),c,1}
\end{align*}
Let $i=1$, and $j\geq 2$:
\begin{align*}
    e_{n-j}p^*_{v,v_{n-j+1}}\alpha^E_{p(x),c,1}&=e_{n-j}e_{n-(j-1)}p^*_{v,v_{n-(j-1)+1}}\\
    &=e_{n-j}\cdots e_{n-1}p^*_{v,v_{n}}\alpha^E_{p(x),c,1}\quad\text{by $(6)$}\\
    &=e_{n-j}\cdots e_{n-1}e_na_p(c)\alpha^E_{p(x),c,1}\\
    &=p^*_{v, v_{n-j}}ca_p(c)\alpha^E_{p(x),c,1}\\
    &=p^*_{v, v_{n-j}}\alpha^E_{p(x),c,1}.
\end{align*}
Let $i\geq 2$, and $j=1$:
\begin{align*}
    e_{n-1}p^*_{v,v_n}\alpha^E_{p(x),c,i}&=e_{n-1}p^*_{v,v_n}\big(ca_p(c)\alpha^E_{p(x),c,i}-\alpha^E_{p(x),c,i-1}\big)\\
    &=e_{n-1}e_na_p(c)\alpha^E_{p(x),c,i}-p^*_{v,v_{n-1}}\alpha^E_{p(x),c,i-1}\\
    &=p^*_{v,v_{n-1}}ca_p(c)\alpha^E_{p(x),c,i}-p^*_{v,v_{n-1}}\alpha^E_{p(x),c,i-1}\\
    &=p^*_{v,v_{n-1}}(\alpha^E_{p(x),c,i}+\alpha^E_{p(x),c,i-1})-p^*_{v,v_{n-1}}\alpha^E_{p(x),c,i-1}\\
    &=p^*_{v,v_{n-1}}\alpha^E_{p(x),c,i}.
\end{align*}
Let $i\geq 2$, and $j\geq 2$:
\begin{align*}
    e_{n-j}p^*_{v,v_{n-j+1}}\alpha^E_{p(x),c,i}&=e_{n-j}p^*_{v,v_{n-j+1}}\big(ca_p(c)\alpha^E_{p(x),c,i}-\alpha^E_{p(x),c,i-1}\big)\\
    &=e_{n-j}\cdots e_n a_p(c)\alpha^E_{p(x),c,i}-p^*_{v,v_{n-j}}\alpha^E_{p(x),c,i-1}\\
    &=p^*_{v,v_{n-j}}ca_p(c)\alpha^E_{p(x),c,i}-p^*_{v,v_{n-j}}\alpha^E_{p(x),c,i-1}\\
    &=p^*_{v,v_{n-j}}(\alpha^E_{p(x),c,i}+\alpha^E_{p(x),c,i-1})-p^*_{v,v_{n-j}}\alpha^E_{p(x),c,i-1}\\
    &=p^*_{v,v_{n-j}}\alpha^E_{p(x),c,i}.  \qedhere
\end{align*}
\end{proof}


\section{The injective envelopes of simple modules \\ associated to exclusive cycles}
\label{section:injenvcycle}

In this, the final section of the article, we prove our main result (Theorem \ref{injhulltheoremc}).   We start by giving the key definition.  

  \begin{definition}
Let $c$ be a cycle in the graph $E$.  Following \cite{AR14}, we call $c$ \emph{exclusive} in case no vertex in $c^0$ is the base of a cycle other than the corresponding cyclic permutations of $c$ itself.  
      \end{definition}

   Our objective is to explicitly construct the injective envelope of the  simple $L_K(E)$-module $V_c^E$ in the situation where $c$ is an exclusive cycle.  

   We note that two similar types of cycles have been studied elsewhere in the literature:  source cycles (no edge in the graph has range in $c^0$ other than the edges of $c$), and maximal cycles (there is no path in $E$ that starts in a cycle $d\neq c$ and ends in $c$).   Easily any source cycle is maximal, and in turn any maximal cycle is exclusive.

\smallskip

\emph{We assume throughout this section that $c=e_1\cdots e_n$ denotes an exclusive cycle, that $v_i$ is the source vertex of the edge $e_i$ for $1\leq i\leq n$, and that $v:=v_1=:v_{n+1} =s(c)$.} 

\smallskip

(At various points throughout this section, for emphasis, we will  remind the reader of this assumption.)

We present in detail the construction of the injective envelope of the simple left $L_K(E)$-module $V^{E}_{p(x),c}$ in the case $p(x)=x-1$, i.e., of the simple module $V^E_{c}$. 
At the end of the section we will then present (without details) the construction of the injective envelope of the simple module $V^{E}_{p(x),c}$ for an arbitrary basic irreducible polynomial $p(x)$; the necessary tools required to establish the general result from the specific  are presented   in \Cref{section:LPA}.  

For each $i\geq 1$,  denote simply by $$\alpha^E_{c,i}$$ the element $\alpha^E_{x-1,c,i}$  of the Pr\"ufer left $L_K(E)$-module $U^E_c$.

For any vertex $w$ in $E$, the tree $T(w)$ is a hereditary subset of $E^0$.   For the vertex $v=s(c)$ of $E$, we  denote by $H$ the set of vertices 
$$H \ := \ T(v)\setminus\{v=v_1,v_2,...,v_n\} \ = \ T(v) \setminus c^0.$$

It will be important for us to consider the full subgraph $T(v)$ of $E$ as a graph in its own right;  for emphasis, we denote this graph by $E_{T(v)}$.

\begin{lemma}\label{HisheredsatinT(V)}   Let $c$ be an exclusive cycle in $E$, and $H$ the set defined above.  

(1)  The set $H$ is a hereditary subset of $E^0$ (and thus also of $E_{T(v)}^0$). 

  (2) The set $H$ is a hereditary and saturated subset of $E_{T(v)}^0$. 
\end{lemma}
\begin{proof} 
(1)   Follows easily from exclusivity,  as no path from a vertex in $T(v)$ outside of $c^0$ can have range in $c^0$.

(2)   Hereditariness has been established in (1).  But no vertex in $c^0$ can be in the saturated closure of $H$, because each vertex $v_i$  of $c^0$ emits an edge not in $H$ (specifically, the edge $e_i$ has range $v_{i+1} \notin H$).  
\end{proof}

\begin{proposition}\label{prop:U^pdescrizione}\label{prop:elementsU}
    Let $c$ be an exclusive cycle in $E$, and $p(x)$ a basic irreducible polynomial in $K[x,x^{-1}]$.    Then the elements of the Pr\"ufer left $L_K(E)$-module $U^E_{p(x),c}$ are finite sums  
    of the form
    \[
   \sum k_{\xi, i, j}  \xi p^*_{v,r(\xi)}c^j\alpha^E_{p(x),c,i}\ ,
    \]
where $ k_{\xi, i, j} \in K,  \xi\in F^+_E(c^0), 0\leq j<\deg p(x),$ and  $i\geq 1.$   

In particular, when $p(x)=x-1$, the elements of the Pr\"ufer left $L_K(E)$-module $U^E_{c}$ are finite sums  of the form
    \[
   \sum k_{\xi, i}  \xi p^*_{v,r(\xi)}\alpha^E_{c,i}\ ,
    \]
where $ k_{\xi, i} \in K,  \xi\in F^+_E(c^0)$, and  $i\geq 1$. 
\end{proposition}
\begin{proof}
The Pr\"ufer left $L_K(E)$-module $U^E_{p(x),c}$ is generated by $\alpha^E_{p(x),c,i}$, $i\geq 1$. It is enough to prove that $\mu \alpha^E_{p(x),c,i}$ has the prescribed form for any monomial $\mu=\mu_1\mu_2^*\in L_K(E)$ such that $\mu \alpha^E_{p(x),c,i}\not=0$. Necessarily $r(\mu)=s(\mu_2)=v$ by \Cref{rem:contialpha}(2). Therefore, all the vertices of the edges in $\mu_2$ belong to the hereditary set $T(v)$. Since $c$ is exclusive, $H:=T(v)\setminus c^0$ is also hereditary in $E^0$ by Lemma \ref{HisheredsatinT(V)}(1). Then we have
     \[
     \mu_2=c^m p_{v,v_x}\nu_1\cdots\nu_t \quad \ell, t\geq 0, 1\leq x\leq n
     \]
     with $v_x=s(\nu_1)\in c^0$, and $r(\nu_1)=s(\nu_2)$, ..., $s(\nu_t)$, $r(\nu_t)\in H$. If $t>0$, then $\nu_1$ is an exit of the cycle $c$ and therefore by \Cref{rem:contialpha}(4,5), we get
     \[0\not=\mu\alpha^E_{p(x),c,i}=\mu_1\mu_2^*\alpha^E_{p(x),c,i}=\mu_1\nu_t^*\cdots \underbrace{\nu_1^* p_{v,v_x}^*(c^*)^m \alpha^E_{p(x),c,i}}_{=0}=0,\]
     a contradiction.
     Then we have
     \[\mu\alpha^E_{p(x),c,i} =\mu_1 p_{v,v_x}^*(c^*)^m \alpha^E_{p(x),c,i}.\]
     By \Cref{rem:contialpha}(4), $(c^*)^m \alpha^E_{p(x),c,i}$ is a finite linear combination $\Xi$ of elements of the form $c^h \alpha^E_{p(x),c,j}$ with $0\leq h<\ell$, and $1\leq j\leq i$.
     If $\mu_1=v_x$, we are done. Otherwise, again because $H$ is hereditary, $r(\mu_1)=v_x\in c^0$ implies that the vertices of the path $\mu_1$ are not in $H$.
     If $\mu_1=\xi_1\cdots \xi_y$, $1\leq y$, let us denote by $\hat z:=\min\{z\mid 1\leq z\leq y, s(\xi_z)\in c^0\}$. Then $\mu_1=\xi_1\cdots \xi_{\hat z-1}p_{r(\xi) v}c^{m'}p_{v,v_x}$ with $\xi:=\xi_1\cdots \xi_{\hat z-1}\in F^+(c^0)$. Observe that $p_{v,v_x}p^*_{v,v_x}c=c$, and $p_{v,v_x}p^*_{v,v_x}\alpha^E_{p(x),c,i}=\alpha^E_{p(x),c,i}$ by \Cref{lemma:nuovo c1c1*}(7). Moreover, if $x^t=q_{s_1}(x)p(x)^{s_1}+\cdots+q_1(x)p(x)+q_0(x)$ in $K[x]$, $\deg q_j(x)<\deg p(x)$, then by \Cref{lemma:nuovo c1c1*}(1),
     \[
     c^t \alpha^E_{p(x),c,i}=q_{s_1}(c)\alpha^E_{p(x),c,i-s_1}+\cdots+q_1(c)\alpha^E_{p(x),c,i-1}+q_0(c)\alpha^E_{p(x),c,i}.
     \]
     Therefore, we get
     \begin{align*}\mu\alpha^E_{p(x),c,i}&=\xi p_{v, r(\xi)}c^{m'}p_{v,v_x}p_{v,v_x}^*\Xi=\xi p_{ v,r(\xi)}^*c\cdot c^{m'}\Xi\\
     &=\xi p_{v,r(\xi)}^*\Xi'
     \end{align*}
     for a suitable finite linear combination $\Xi'$ of elements of the form $c^h \alpha^E_{p(x),c,j}$ with $0\leq h<\ell$, and $1\leq j\leq i$.
\end{proof}

We establish the desired description of the injective envelope of $V^E_{c}$ (Theorem \ref{injhulltheoremc}) by first establishing two major pieces.  In the First Piece,   
we build the ideal $I(v)$ of $L_K(E)$ generated by $v$; $I(v)$ is a ring in its own right, and the $L_K(E)$-modules $U_c^E$ and $V^E_{c}$ are in fact modules over $I(v)$.   We show (Corollary  \ref{cor;Pruferiniettivo})   that $U_c^E$ is the injective envelope of  $V^E_{c}$ as $I(v)$-modules. Establishing this fact will require a tour through various key ideas in Leavitt path algebras, including the realization of a specific quotient of a  Leavitt path algebra of a (possibly infinite) graph as the Leavitt path algebra of an associated (finite) graph, as well as the Morita equivalence result Theorem \ref{thm:restricted_hedgehog_Morita}.  In the Second Piece we explicitly build (Definition \ref{def;KF+0}) a $K$-vector space  (a space of ``formal series''), and then work quite hard to show that an appropriate $K$-subspace of this $K$-space indeed admits a left $L_K(E)$-module structure (Proposition \ref{prop:formalisTmod}), in which   $V^E_{c}$ lives as an essential submodule.    With these two key pieces in place, we then establish our main result (Theorem \ref{injhulltheoremc}), which asserts that this $L_K(E)$-module is indeed the injective envelope of the $L_K(E)$-module $V^E_{c}$.


\subsection{The First Piece, in which  we show that $U_c^E$ is the injective envelope of  $V^E_{c}$, viewed as modules over the two-sided ideal $I(v)$ of $L_K(E)$ generated by $v$. }

Consider the Leavitt path algebra $L_K(E_{T(v)})$ associated to the restricted graph $E_{T(v)}$,  and let $I(H)$ denote the two-sided ideal of $L_K(E_{T(v)})$ generated by $H = T(v) \setminus c^0$. 
Since by Lemma \ref{HisheredsatinT(V)}(2) $H$ is a hereditary and saturated subset in $E_{T(v)}$, invoking \cite[Theorems 1.5.18, 2.4.12, 2.4.15]{AAM} we get that the quotient $K$-algebra $L_K(E_{T(v)})/I(H)$ is isomorphic to the Leavitt path algebra $L_K(F)$, where $F:= E_{T(v)}/(H,\emptyset)$ is the graph given by:
\begin{align*}
F^0&=c^0\sqcup \{u'\mid u\in B_H\}=c^0\sqcup \{v'_i\mid v_i\in \text{Inf}(E_{T(v)}),\ 1\leq i\leq n\}\\
F^1&=\{e_1,..., e_n\}\sqcup \{e_i'\mid r(e_i)=v_{i+1}\in\text{Inf}(E_{T(v)}),\ 1\leq i\leq n\}\\
s_{F}(e_i)&=v_i,  \ \ \ r_{F}(e_i) = v_{i+1}, \ \ \text{and} \ \\
 s_{F}(e_i') &= v_{i}, \ \  r_{F}(e_i') = v_{i+1}' \ \ \text{whenever} \ v_{i+1} \in \text{Inf}(E_{T(v)}).
\end{align*}
Clearly $F$ is a finite graph containing $c$ as a source cycle.
Using the explicit descriptions of the isomorphisms given in the results of \cite{AAM} cited above,  the composition 
$$\Phi: L_K(E_{T(v)})\to L_K(F)$$
 of the projection $L_K(E_{T(v)})\to L_K(E_{T(v)})/I(H)$ with the aforementioned isomorphism $L_K(E_{T(v)})/I(H)\to L_K(F)$ is defined in the following way on the vertices, real edges,  and ghost edges of the cycle $c$:
\begin{align*}\text{If $v_i\in {\rm Inf}(E_{T(v)})$, then }&\Phi(v_i)=v_i+v_i'\text{ for $i=1,...,n$};\\
\text{If $v_i\in {\rm Reg}(E_{T(v)})$, then }&\Phi(v_i)=v_i\text{ for $i=1,...,n$};\\
\text{If $r(e_i) \in {\rm Inf}(E_{T(v)})$, then }&\Phi(e_i)=e_i+e_i'\text{ for $i=1,...,n$};\\
\text{If $r(e_i) \in {\rm Reg}(E_{T(v)})$, then }&\Phi(e_i)=e_i\text{ for $i=1,...,n$};\\
\text{If $r(e_i) \in {\rm Inf}(E_{T(v)})$, then }&\Phi(e_i^*)=e_i^*+(e_i')^*\text{ for $i=1,...,n$};\\
\text{If $r(e_i) \in {\rm Reg}(E_{T(v)})$, then }&\Phi(e_i^*)=e_i^*\text{ for $i=1,...,n$}.\end{align*}

\begin{lemma}\label{lemma:Phi(e_ie_i^*)}  Let $\Phi$ be the ring homomorphism defined above. 
\begin{enumerate}
    \item For each $1\leq i\leq n$ we have 
    $\Phi(e_ie_i^*) = v_i$.
    
    \item If $v\in {\rm Reg}(E_{T(v)})$, then $\Phi(c) = c$.  Otherwise, if $v\in {\rm Inf}(E_{T(v)})$, then $\Phi(c) =e_1 \cdots e_{n-1}(e_n+e'_n)$.  

    \item  For each $1\leq i\leq n$ we have  $\Phi((c-v)^i e_1e_1^\ast) = (c-v)^i$.
    \end{enumerate}
\end{lemma}
\begin{proof}
   1.  If $v_{i+1}=r(e_i)$ is not an infinite emitter in $E_{T(v)}$, then $s_F^{-1}(v_i)=\{e_i\}$ and hence by the (CK2) relation applied in $L_K(F)$
    \[
    \Phi(e_ie_i^*)=\Phi(e_i)\Phi(e_i^*)=e_ie_i^*=v_i.
    \]
    If $v_{i+1}=r(e_i)$ is an infinite emitter, then $s_F^{-1}(v_i)=\{e_i, e_i'\}$ and hence again by the (CK2) relation applied in $L_K(F)$
    \[
    \Phi(e_ie_i^*) = (e_i + e_i')(e_i^* + (e_i')^*) = e_ie_i^* + 0 + 0 + e_i'(e_i')^* = v_i.
    \]

    2. In the expression $\Phi(c) = \Phi(e_1 \cdots e_n) =   \Phi(e_1) \cdots \Phi(e_{n-1})\Phi(e_n)$, any term arising of the form $e_{j}' e_{j+1} $ is $0$ (since $r(e'_{j}) \neq s(e_{j+1})$ for any $1\leq j<n-1$).
   Thus the only situation in which $\Phi(c) \neq c$ is when $v$ is an infinite emitter. In such a case, again because $e_j'e_{j+1}=0
   $ for any $1\leq j<n-1$, we have $\Phi(c) =e_1 \cdots e_{n-1}(e_n+e'_n)$. 

   3.  If $v$ is not an infinite emitter in $E_{T(v)}$ the result is clear.  If $v$ is an infinite emitter in $E_{T(v)}$ then (using statement 1) 
   $$\Phi((c-v)^i e_1e_1^\ast) = (e_1 \cdots e_{n-1}(e_n+e'_n) - (v + v'))^i  v = (c-v)^i,$$
   because $e'_nv = v'v = 0$.  
\end{proof}

Invoking \cite[Theorem 6.4]{AMT19} (which requires that $F$ be a finite graph and that $c$ be a maximal cycle in $F$),  we have 
\begin{proposition}\label{PruferoverF}  With $F$ defined as above, the injective envelope of the simple left $L_K(F)$-module $V^F_{c}$ 
is the Pr\"ufer $L_K(F)$-module $U^F_c$.
\end{proposition}

(Since any source cycle is maximal, \cite[Theorem 6.4]{AMT19} indeed applies to the cycle $c$ within $F$.)   We emphasize that \cite[Theorem 6.4]{AMT19} is a fairly deep and intricate result;  thus Proposition \ref{PruferoverF} should  be viewed as highly nontrivial.

Through the ring homomorphism $\Phi:L_K(E_{T(v)})\to L_K(F)$, every left $L_K(F)$-module $M$ becomes a left $L_K(E_{T(v)})$-module by setting:
\[
x\cdot m:=\Phi(x)\cdot m\quad \forall m\in M,\forall x\in L_K(E_{T(v)}).
\]
So combining Proposition \ref{PruferoverF} with   \Cref{lemma:injectivelocalunits} we have

\begin{corollary}\label{alsoasLKET(V)modules} The injective envelope of the simple left $L_K(E_{T(v)})$-module $V^F_{c}$ is the left  $L_K(E_{T(v)})$-module $U^F_c$. 
\end{corollary}


We now show that the left $L_K(E_{T(v)})$-modules  $V^F_{c}$ and $U^F_c$ are isomorphic to 
$V^{E_{T(v)}}_{c}$ and $U^{E_{T(v)}}_c$, respectively. By Corollary \ref{alsoasLKET(V)modules}, this will yield that $U^{E_{T(v)}}_c$ is the injective envelope of $V^{E_{T(v)}}_{c}$ in $L_K(E_{T(v)})\Mod$.  We start by establishing

\begin{lemma}\label{lemma:I(H)U=0}
    Let $I(H)$ denote the two-sided ideal of $L_K(E_{T(v)})$ generated by $H$.  Then $xu=0$ for every $x\in I(H)$ and $u \in U^{E_{T(v)}}_c$. 
\end{lemma}
\begin{proof}
    By \cite[Lemma 2.4.1]{AAM}, every element of $I(H)$ is of the form $\sum_{j=1}^n k_j\gamma_j\lambda_j^*$ with $n\geq 1$, $\gamma_j,\lambda_j\in\Path(E_{T(v)})$,  $r(\gamma_j)=r(\lambda_j)\in H$ for $1\leq j\leq n$. Any element of $U^{E_{T(v)}}_c$ belongs to $L_K(E_{T(v)})\alpha^{E_{T(v)}}_{c,i}$ for a suitable $i\geq 1$. Since $I(H)L_K(E_{T(v)})\alpha^{E_{T(v)}}_{c,i}=I(H)\alpha^{E_{T(v)}}_{c,i}$, we have to prove that 
    \[
    \sum_{j=1}^n k_j\gamma_j\lambda_j^*\alpha^{E_{T(v)}}_{c,i}=0\quad\forall i\geq 1.
    \]
     By \Cref{rem:contialpha}(2), we can assume $s(\lambda_j)=r(\lambda_j^*)=v$ for $1\leq j\leq n$. Moreover, note that $\lambda_j \neq v$, since $v\notin H$;  i.e., each $\lambda_j$ has length $\geq 1$. 
     
     
     Fix $1\leq j\leq n$; we have  $\lambda_{j}=\lambda_{j1}\cdots \lambda_{j\ell_{j}}$ for suitable $ \lambda_{j1},\dots, \lambda_{j\ell_{j}}\in E_{T(v)}^1$. Set 
    \[\hat i=\min\{h:r(\lambda_{jh})\in H\}.\]
    Then $\lambda_{j\hat i}$ is an exit for $c$. By \Cref{rem:contialpha}(4,5), there exists $m\geq 0$ such that
    \[
    \lambda_{j\hat i}^*\cdots \lambda_{j1}^*\alpha^{E_{T(v)}}_{c,i}=\lambda_{j\hat i}^*p^*_{v s(\lambda_{j\hat i})}(c^*)^m\alpha^{E_{T(v)}}_{c,i}=0.
    \]
    As  $j$ was arbitrary, we conclude $ \sum_{j=1}^n k_j\gamma_j\lambda_j^*\alpha^{E_{T(v)}}_{c,i}=0$.
\end{proof}

\begin{proposition}\label{prop:isoETvF}
    For each $i\geq 1$,   
    $$L_K(E_{T(v)})\alpha_{c,i}^{E_{T(v)}} \cong  L_K(E_{T(v)})\alpha_{c,i}^F$$
    as  left $L_K(E_{T(v)})$-modules.
    In particular, 
  $$V_{c}^F \cong L_K(E_{T(v)})\alpha_{c,1}^F \cong L_K(E_{T(v)})\alpha_{c,1}^{E_{T(v)}} \cong V_{c}^{E_{T(v)}}$$
  as left $L_K(E_{T(v)})$-modules.  
\end{proposition}

\begin{proof}
    Let $\lambda\in L_K(E_{T(v)})$. Then $\lambda=\lambda_1+\lambda_2$ where
    $\lambda_1=\lambda_1 v$, and $\lambda_2$ is a sum of monomials whose range is not $v$ (i.e., $\lambda_2v=0$). Assume $\lambda \alpha_{c,i}^{E_{T(v)}}=0$;  we  prove that also $\lambda \cdot \alpha_{c,i}^F:=\Phi(\lambda)\alpha_{c,i}^F=0$, where $\Phi: L_K(E_{T(v)}) \to L_K(F)$ is the ring homomorphism defined above. By Lemma~\ref{rem:contialpha}
    \[0=(\lambda_1+\lambda_2)\alpha_{c,i}^{E_{T(v)}}=\lambda_1\alpha_{c,i}^{E_{T(v)}}=\lambda_1\big(v+L_K(E_{T(v)})(c-v)^i\big).\]
    Therefore $\lambda_1 \in L_K(E_{T(v)})(c-v)^i$;  write
    $\lambda_1 = \mu_1(c-v)^i$ for some $\mu_1 \in L_K(E_{T(v)}).$

    By Lemma \ref{lemma:Phi(e_ie_i^*)}   we then have 
    \begin{align*}
        \Phi(\lambda_1)&=\Phi(\mu_1(c-v)^i)=\Phi(\mu_1)\Phi((c-v)^i)\\
        &=\begin{cases}
            \Phi(\mu_1)(c-v)^i&\text{ if $v\in \text{Reg}(E_{T(v)})$};\\
            \Phi(\mu_1)(e_1\cdots e_{n-1}(e_n+e_n')-(v+v'))^i&\text{ if $v\in \text{Inf}(E_{T(v)})$.} 
        \end{cases}
    \end{align*}
In the first case we have $\Phi(\lambda_1) \alpha_{c,i}^F = 0$ because $(c-v)^i \alpha_{c,i}^F = 0$.\\
For the second case: 
\begin{align*}
\Phi(\lambda_1) \alpha_{c,i}^F&=   \Phi(\mu_1) (e_1\cdots e_{n-1}(e_n+e_n')-(v+v'))^i \cdot \alpha_{c,i}^F\\
&=\Phi(\mu_1)(e_1\cdots e_{n-1}(e_n+e_n')-(v+v'))^{i-1} (c - v) \cdot \alpha_{c,i}^F\\
&\text{\ \ \ \ \ \ \ \ \ \ (because $e_n'\cdot \alpha_{c,i}^F=0=v' \cdot \alpha_{c,i}^F$) }\\
&=\Phi(\mu_1)(e_1\cdots e_{n-1}(e_n+e_n')-(v+v'))^{i-1} \cdot \alpha_{c,i-1}^F.
\end{align*}
Continuing in this way we get that 

\smallskip

$\Phi(\lambda_1) \alpha_{c,i}^F  =  \Phi(\mu_1)(e_1\cdots e_{n-1}(e_n+e_n')-(v+v'))^{1} \cdot \alpha_{c,1}^F  $

\qquad \qquad $=  \Phi(\mu_1)(c-v)\alpha_{c,1}^F  =   0 .$

\smallskip

So we conclude in either case that
 $\Phi(\lambda_1)\alpha_{c,i}^F=0$.\\
    Next, since $\Phi(e_1e_1^*)=v$ (see \Cref{lemma:Phi(e_ie_i^*)}) and $\lambda_2v=0$,
    \[\Phi(\lambda_2)\alpha_{c,i}^F=
    \Phi(\lambda_2)v\alpha_{c,i}^F=\Phi(\lambda_2e_1e_1^*)\alpha_{c,i}^F = \Phi(\lambda_2ve_1e_1^\ast)\alpha_{c,i}^F = \Phi(0)\alpha_{c,i}^F=0.
    \]
The above discussion shows that  setting $x \alpha_{c,i}^{E_{T(v)}}\mapsto x \cdot \alpha_{c,i}^F :=\Phi(x)\alpha_{c,i}^F $, for each $x\in L_K(E_{T(v)})$, gives a well-defined  
$L_K(E_{T(v)})$-homomorphism 
    $$\phi: \ L_K(E_{T(v)})\alpha_{c,i}^{E_{T(v)}} \ \to \ L_K(E_{T(v)})\alpha_{c,i}^F$$
    of $L_K(E_{T(v)})$-modules. 
    Clearly $\phi$ is surjective. 
    Let us prove that $\phi$ is injective. Suppose that for $\lambda=\lambda v\in L_K(E_{T(v)})$ we have 
    \[0=\phi(\lambda\alpha_{c,i}^{E_{T(v)}})=\lambda \cdot \alpha_{c,i}^F := \Phi(\lambda)\alpha_{c,i}^F=\Phi(\lambda)(v+L_K(F)(c-v)^i).
    \]
    Then  $\Phi(\lambda)v\in L_K(F)(c-v)^i$; write $\Phi(\lambda)v = \xi (c-v)^i$ for a suitable $\xi\in L_K(F)$.  Now  by \Cref{lemma:Phi(e_ie_i^*)}
    \begin{align*}
        \Phi(\lambda)v&=\Phi(\lambda)\Phi(e_1e_1^*)=\Phi(\lambda e_1e_1^*),
    \end{align*}
    so that $ \Phi(\lambda e_1e_1^*) = \xi (c-v)^i$.  Since $\Phi((c-v)^ie_1e_1^*)=(c-v)^i$ (Lemma \ref{lemma:Phi(e_ie_i^*)}), setting $\xi=\Phi(\mu)$ for a suitable $\mu\in L_K(E_{T(v)})$ we get
    \[\Phi(\lambda e_1e_1^*-\mu (c-v)^ie_1e_1^*)=0\quad\text{and hence}\]
    $(\lambda -\mu (c-v)^i)e_1e_1^*$ belongs to $\Ker\Phi=I(H)$.
Therefore, by \Cref{lemma:I(H)U=0} and \Cref{lemma:nuovo c1c1*}(7) we get
\[0=(\lambda-\mu(c-v)^i)e_1e_1^*\alpha_{c,i}^{E_{T(v)}}=(\lambda-\mu(c-v)^i)\alpha_{c,i}^{E_{T(v)}}=\lambda \alpha_{c,i}^{E_{T(v)}}. \qedhere \]
    
\end{proof}

\begin{corollary}\label{PruferisenvelopeoverT(v)}
    The Pr\"ufer left $L_K(E_{T(v)})$-module $U^{E_{
T(v)}}_c$ is the injective envelope of the simple left $L_K(E_{T(v)})$-module $V^{E_{T(v)}}_{c}$.
\end{corollary}
\begin{proof}
From  \Cref{prop:isoETvF} we get that for any $i\geq 1$ there exists an isomorphism $\phi_i:L_K(E_{T(v)})\alpha_{c,i}^{E_{T(v)}}\to L_K(E_{T(v)})\alpha_{c,i}^F$. Denote by $\psi_i$ the composition of $\phi_i$ and the inclusion of left $L_K(E_{T(v)})$-modules $L_K(E_{T(v)})\alpha_{c,i}^F\hookrightarrow U^F_c$. Then 
$$\varinjlim_i \psi_i:U^{E_{
T(v)}}_c\to U^F_c$$ is an isomorphism.
Since by Corollary \ref{alsoasLKET(V)modules}  $U^F_c$ is the injective envelope of the simple left $L_K(E_{T(v)})$-module $V^F_{c}\cong L_K(E_{T(v)})\alpha_{c,1}^F$, which in turn by the second statement of  Proposition \ref{prop:isoETvF}  is isomorphic to $V^{E_{T(v)}}_{c}\cong L_K(E_{T(v)})\alpha_{c,1}^{E_{T(v)}}$, the result follows.
\end{proof}

The upshot of Corollary \ref{PruferisenvelopeoverT(v)} is that we have explicitly constructed the injective envelope of the simple module $V^{E_{T(v)}}_{c}$, viewed as a module over $L_K(E_{T(v)})$.  We will now use this description to identify the injective envelope of $V^{E}_{c}$ when viewed as a module over the two-sided ideal of $L_K(E)$ generated by $v$.


By \cite[Proposition 2.2.22]{AAM} we have an embedding of $K$-algebras
\[L_K(E_{T(v)})\hookrightarrow L_K(E).\]
Indeed, by Lemma \ref{lem:restrictedversushedgehog}, $$L_K(E_{T(v)}) = \Sig (T(v)) L_K(E) \Sig (T(v)).$$

\begin{lemma}\label{lemmakillalpha}
    For each $i\geq 1$ and $\tau\in L_K(E_{T(v)})$, we have
    \[
    \tau\alpha_{c,i}^{E_{T(v)}}=0\text{ if and only if }\tau\alpha_{c,i}^E=0.
    \]
\end{lemma}
\begin{proof}
We may assume $\tau v = \tau$.    Let $\rho\in\Sig( T(v))$ such that $\rho\tau=\tau$. We have
\[\tau \alpha_{c,i}^E=\tau (v+L_K(E)(c-v)^i)=0  \ \  \Longleftrightarrow \ \  
\tau  \in L_K(E)(c-v)^i.\] 
Since $\tau = \rho \tau$,
this happens if and only if $\tau$ belongs to $\Sig ( T(v))L_K(E)\Sig ( T(v))(c-v)^i$, which by Lemma \ref{lem:restrictedversushedgehog}(1) equals $L_K(E_{T(v)})(c-v)^i.$
The latter happens if and only if $\tau\alpha_{c,i}^{E_{T(v)}}=0$.
\end{proof}

Consider the two-sided ideal $$I(v) = I(T(v))$$
generated by $v$ in the $K$-algebra $L_K(E)$.  Then $I(v)$ is a ring with local units.   So by  \Cref{prop:newMorita} the $K$-algebras \[ I(v)=L_K(E)\Sig ( T(v))L_K(E) \text{ and }L_K(E_{T(v)})=\Sig ( T(v)) L_K(E)\Sig ( T(v))\] are Morita equivalent. The functors realizing the Morita equivalence are
\[I(v)\Mod\overset{F}{\underset{G}{\rightleftarrows}} L_K(E_{T(v)})\Mod\]
where 
\[F:=\displaystyle\varinjlim_{\rho\in \Sig ( T(v))}\rho I(v)\otimes_{I(v)} \ \underline{\ \ \ \  }  \ \text{ and } \ G:=\displaystyle\varinjlim_{\rho\in \Sig ( T(v))}I(v)\rho\otimes_{\rho I(v)\rho} \ \underline{\ \ \ \  }.\]

\begin{remark}\label{VandUareI(v)modules}
We note that since $v\alpha_{c,i}^E = \alpha_{c,i}^E$ for all $i\geq 1$, and since $\{\alpha_{c,i}^E \ | \ i\in \mathbb{N}\}$ is an $L_K(E)$-generating set for $U_c^E$, we get that both $V_c^E$ and $U_c^E$ are indeed $I(v)$-modules.
\end{remark}

\begin{proposition}\label{isoUEcGUEc}
Denote by $\iota^E_i$ (resp. $\iota^{E_{T(v)}}_i$) the embedding
\[
L_K(E)\alpha^E_{c,i}\hookrightarrow L_K(E)\alpha^E_{c,i+1},\ \alpha^E_{c,i}\mapsto (c-v)\alpha^E_{c,i+1}
\]
\[
\Big(\text{resp. }L_K(E_{T(v)})\alpha^{E_{T(v)}}_{c,i}\hookrightarrow L_K(E_{T(v)})\alpha^{E_{T(v)}}_{c,i+1},\ \alpha^{E_{T(v)}}_{c,i}\mapsto (c-v)\alpha^{E_{T(v)}}_{c,i+1}\Big).
\]
There exist isomorphisms of $I(v)$-modules
\[\varphi_i:I(v)\alpha^E_{c,i}\to G(L_K(E_{T(v)})\alpha^{E_{T(v)}}_{c,i})\]
such that $\varphi_{i+1}\circ \iota^E_i= G(\iota^{E_{T(v)}}_i)\circ\varphi_i$.
In particular we have  isomorphisms of left $I(v)$-modules
\[V^E_{c} \cong G(V^{E_{T(v)}}_{c})\quad\text{and}\quad U^E_c \cong  G(U^{E_{T(v)}}_c).\]
\end{proposition}
\begin{proof}
  For any subset $S$ of $\Sig(T(v))$,  let  
\[ \Sig(T(v))^{\geq S} \ \mbox{ denote } \{\rho \in \Sig(T(v)) \ | \ \rho \geq s \ \forall s\in S\}. \]    
Assume $\lambda \alpha^E_{c,i}=0$ with $\lambda=\lambda v\in I(v)$; then $\lambda=\lambda'(c-v)^i$ for some $\lambda'=\lambda'v\in I(v)$. Then for each $ \rho\in\Sig(T(v))^{\geq v},$
\begin{align*}\lambda \otimes_{\rho L_K(E)\rho} \alpha^{E_{T(v)}}_{c,i}&=\lambda'(c-v)^i\otimes_{\rho L_K(E)\rho} \alpha^{E_{T(v)}}_{c,i}\\
&=\lambda'\otimes_{\rho L_K(E)\rho}(c-v)^i\alpha^{E_{T(v)}}_{c,i}=0.
\end{align*}
Therefore, for each $\rho \in \Sig(T(v))^{\geq v}$,  setting $\alpha^E_{c,i}\mapsto v\otimes_{\rho L_K(E)\rho}\alpha^{E_{T(v)}}_{c,i}$ we define homomorphisms 
\[\varphi_{i,\rho}:I(v)\alpha^E_{c,i}\to I(v)\rho \otimes_{\rho L_K(E)\rho} \rho L_K(E_{T(v)})\alpha^{E_{T(v)}}_{c,i}\]
\[  \mbox{ via }  \ \ \lambda \alpha^E_{c,i} =\lambda v \alpha^E_{c,i} \ \mapsto \ 
\lambda v \otimes_{\rho L_K(E)\rho}\alpha^{E_{T(v)}}_{c,i}.\quad 
\]
The homomorphisms $\varphi_{i,\rho}$ are surjective because
\[I(v)\rho \otimes_{\rho L_K(E)\rho} \rho L_K(E_{T(v)})\alpha^{E_{T(v)}}_{c,i}=I(v)v\otimes_{\rho L_K(E)\rho} \alpha^{E_{T(v)}}_{c,i}.\]
Therefore \begin{align*}\varphi_i:=\varinjlim_{\rho \in \Sig(T(v))^{\geq v}} \varphi_{i,\rho}&: I(v)\alpha^E_{c,i}\to 
\varinjlim_{\rho \in \Sig(T(v))^{\geq v}}I(v)\rho\otimes_{\rho L_K(E)\rho}\rho L_K(E_{T(v)})\alpha^{E_{T(v)}}_{c,i}\\
&\mbox{i.e.,} \ \ \varphi_i: \ I(v)\alpha^E_{c,i}\to  G(L_K(E_{T(v)})\alpha^{E_{T(v)}}_{c,i})\end{align*}
is an epimorphism.\\

We now show that $\varphi_i$ is a monomorphism. If $\varphi_i(\lambda \alpha^E_{c,i})=0$ with $\lambda\in I(v)$, then there exists $ \overline\rho\in \Sig (T(v))^{\geq v}$ such that 
\[0=\varphi_{i,\rho}(\lambda \alpha^E_{c,i})=\lambda v\otimes_{\rho L_K(E)\rho} \alpha^{E_{T(v)}}_{c,i}\quad\forall \rho \in (\Sig(T(v)))^{\geq \overline{\rho}}. \]
By \cite[Lemma 2.4.1]{AAM}, 
\[\lambda v=\sum_{j=1}^m k_j\gamma_j\delta_j^*v \]
where  $k_j \in K, \gamma_j,\delta_j\in\Path(E),$ and $r(\gamma_j)=r(\delta_j)\in T(v).$
Since $r(\gamma_j)\in T(v)$ for all $1\leq j \leq m$, we have 
\[\gamma_j=\gamma_j'\gamma_j''\quad \text{with } \gamma_j'\in F^+_E(T(v)) \mbox{ and }  \gamma_j''\in {\rm Path}(E_{T(v)}). \]

 Let $\mathcal{L}$ denote a subset of $\{1,2, \dots , m\}$ for which $\{\gamma'_\ell \ | \ \ell \in \mathcal{L}\}$ represents exactly one copy of each of the elements of the set  $\{\gamma'_j \ | \ 1 
\leq j \leq m\}$.    \ \ Now fix $\ell \in \mathcal{L}$.  

By Lemma \ref{lemma:mu*lambda}, for each $1\leq j \leq m$,  
\[
\gamma_\ell'(\gamma_\ell')^*\gamma_j'=\begin{cases}
 \gamma_\ell'     & \text{if }\gamma_j'=\gamma_\ell', \\
    0  & \text{otherwise}.
\end{cases}
\]
So multiplying the displayed expression for $\lambda v$ by $\gamma_\ell'(\gamma_\ell')^\ast$ we get 
\[
\gamma_\ell'(\gamma_\ell')^*\lambda v=\sum_{j: \gamma_j'=\gamma_\ell'} k_j \gamma_\ell'\gamma_{j}''\delta_{j}^* v.
\]
We will show that  $\sum_{j: \gamma_j'=\gamma_\ell'} k_j \gamma_\ell'\gamma_{j}''\delta_{j}^* v \alpha^E_{c,i}=0.$
Since
\[ \lambda \alpha_{c,i}^E  \ =   \   \sum_{j=1}^m k_j\gamma_j\delta_j^*v \alpha_{c,i}^E \ = \ \sum_{\ell \in \mathcal{L}}  ( \sum_{j: \gamma_j'=\gamma_\ell'} k_j \gamma_\ell'\gamma_{j}''\delta_{j}^* v \alpha^E_{c,i}), \] 
the desired result will follow.

 Let $\rho \geq \overline{\rho}$; so  $0=\varphi_{i,\rho}(\lambda \alpha^E_{c,i})=\lambda v\otimes_{\rho L_K(E)\rho} \alpha^{E_{T(v)}}_{c,i}  $.
Now  multiplying  by $(\gamma_\ell')^*$ and using the previous display we get
\begin{align*}
0&=(\gamma_\ell')^*\varphi_{i,\rho}(\lambda \alpha^E_{c,i})=(\gamma_\ell')^*\lambda v\otimes_{\rho L_K(E)\rho} \alpha^{E_{T(v)}}_{c,i}\\
&=\sum_{j: \gamma_j'=\gamma_\ell'} k_j r(\gamma_\ell')\gamma_{j}''\delta_{j}^* v\otimes_{\rho L_K(E)\rho}\alpha^{E_{T(v)}}_{c,i}.
\end{align*}

Let $\mathcal{S}$ denote the set $\{\overline{\rho} \} \cup \{r(\gamma_\ell') \ | \ \ell \in \mathcal{L}\}$.   Then for any $\rho \in \Sig(T(v))^{\geq \mathcal{S}}$ we have $
 k_j r(\gamma_\ell')\gamma_{j}''\delta_{j}^* v \in \rho L_K(E) \rho,$ so that 
\[0=\sum_{j: \gamma_j'=\gamma_\ell'} k_j r(\gamma_\ell')\gamma_{j}''\delta_{j}^* v\otimes_{\rho L_K(E)\rho}\alpha^{E_{T(v)}}_{c,i}=\rho\otimes_{\rho L_K(E)\rho}\sum_{j: \gamma_j'=\gamma_\ell'} k_j r(\gamma_\ell')\gamma_{j}''\delta_{j}^* v \alpha^{E_{T(v)}}_{c,i},\]
and hence $\sum_{j: \gamma_j'=\gamma_\ell'} k_j r(\gamma_\ell')\gamma_{j}''\delta_{j}^* v \alpha^{E_{T(v)}}_{c,i}=0$.
Therefore, \[\sum_{j: \gamma_j'=\gamma_\ell'} k_j r(\gamma_\ell')\gamma_{j}''\delta_{j}^* v \in L_K(E_{T(v)})(c-v)^i \ \ \text{ for each } \ell \in \mathcal{L},\]
and so 
\[ \sum_{j: \gamma_j'=\gamma_\ell'} k_j r(\gamma_\ell')\gamma_{j}''\delta_{j}^* v \alpha^E_{c,i} \ = \ 0 \]
as $L_K(E_{T(v)}) \subseteq L_K(E)$.  So  multiplying by $\gamma_\ell'$ we get 
\[   
    0 = \gamma_\ell'\sum_{j: \gamma_j'=\gamma_\ell'} k_j r(\gamma_\ell')\gamma_{j}''\delta_{j}^* v \alpha^E_{c,i}
    =\sum_{j: \gamma_j'=\gamma_\ell'} k_j \gamma_\ell'\gamma_{j}''\delta_{j}^* v \alpha^E_{c,i},
\]
which, as indicated above, yields the desired conclusion that $\varphi_i$ is a monomorphism, and thus we have shown that each $\varphi_i$ is an isomorphism. 
 In particular,
$$ \varphi_1: I(v)\alpha_{c,1}^E \to G(L_K(E_{T(v)})\alpha^{E_{T(v)}}_{c,1}) $$
is an isomorphism, so that 
$$I(v)\alpha_{c,1}^E = V^E_{c} \cong G(V^{E_{T(v)}}_{c}).$$
(The equality in the display  follows as $I(v)\alpha_{c,1}^E $ is a nonzero $L_K(E)$-submodule of the simple $L_K(E)$-module $V^E_{c}$.)

Because the isomorphisms $\{ \varphi_i$  |   $i\geq 1\}$ are explicitly described, it is straightforward to check that $\varphi_{i+1}\circ \iota^E_i= G(\iota^{E_{T(v)}}_i)\circ\varphi_i$. Thus taking the limit we conclude that $U^E_c$ is isomorphic to $G(U_c^{E_{T(v)}})$ as left $I(v)$-modules.
\end{proof}

We are now in position to establish the First Piece.  

\begin{corollary}\label{cor;Pruferiniettivo}
    Let $E$ be any graph, and $c$ an exclusive cycle in $E$.  Then the Pr\"ufer $I(v)$-module $U^E_c$ is the injective envelope in $I(v)\Mod$ of the simple $I(v)$-module $V^E_{c}$.
\end{corollary}
\begin{proof}
Any Morita equivalence transforms injective modules to injective modules. Since by \Cref{isoUEcGUEc} $U^E_c\cong G(U_c^{E_{T(v)}})$, and $V^E_{c}\cong G(V^{E_{T(v)}}_{c})$ in $I(v)\Mod$, the statement follows.
\end{proof}

\subsection{The Second Piece, in which we define the $L_K(E)$-module $X_c^E$ that will be shown to be the injective envelope in $L_K(E)$-Mod  of $V_{c}^E$}

We denote by $\mathbb N$ the set of natural numbers $\{ 1,2,3, \dots\}$.  

Consider the subset $c^0$ of $E^0$ consisting of the vertices $v=v_1, v_2, \dots, v_n$ of the exclusive cycle $c$.   

The map $F^+_E(c^0)\times \mathbb N\to U^E_c$, $(\mu,i)\mapsto \mu p^*_{v,r(\mu)}\alpha^E_{c,i}$ is  injective (see Remark \ref{rem:description} below). With that in mind, following what we observed at the end of \Cref{section:associativering} regarding formal series,  we give the following

\begin{definition}\label{Mcdef}
We denote by $$K[[F^+_E(c^0)[\mathbb N ]]]$$ the $K$-vector space of all functions $\mathfrak p$ from $F^+_E(c^0)\times \mathbb N $ to $K$ that assume for each $\mu\in F^+_E(c^0)$ the value 0 for almost all $i\in \mathbb N $. We represent any element $\mathfrak p\in K[[F^+_E(c^0)[\mathbb N ]]]$ as a ``\emph{formal series}''
\[\mathfrak p=\sum_{(\mu,i)\in F^+_E(c^0)\times \mathbb N }k_{\mu, i} \mu p^*_{v,r(\mu)}\alpha^E_{c,i},\]
where $k_{\mu, i}=\mathfrak p(\mu,i)$ and $\alpha^E_{c,i}$ are the generators of the Pr\"ufer left $L_K(E)$-module $U^E_c$. 
\end{definition}

\begin{remark}\label{rem:description}
Observe that each pair $(\mu,i)\in F^+_E(c^0)\times \mathbb N $ uniquely determines the element $\mu p^*_{v,r(\mu)}\alpha^E_{c,i}\in U^E_c$. Indeed, if \[\mu_1 p^*_{v,r(\mu_1)}\alpha^E_{c,i_1}=\mu_2 p^*_{v,r(\mu_2)}\alpha^E_{c,i_2},\] then, multiplying on the left by $p_{v,r(\mu_1)}\mu_1^*$, by \Cref{lemma:mu*lambda,lemma:nuovo c1c1*} we get $\mu_1=\mu_2$, which subsequently gives 
$v\alpha^E_{c,i_1}=v\alpha^E_{c,i_2}$, which implies $i_1=i_2.$

 Thus, it is not inappropriate to use $\mu p^*_{v,r(\mu)}\alpha^E_{c,i}$ as a placeholder for the scalar $k_{\mu,i}$. This notation allows us to naturally view the Pr\"ufer left $L_K(E)$-module $U^E_c$ as a $K$-subspace of the $K$-vector space $K[[F^+_E(c^0)[\mathbb N ]]]$ (see \Cref{prop:elementsU}).
  \end{remark}

\begin{definition}\label{def;KF+0}
    We denote by $$K[[F^+_E(c^0)[\mathbb N ]]]^{sf}$$ the $K$-vector subspace of $K[[F^+_E(c^0)[\mathbb N ]]]$ consisting of all \emph{source finite} functions 
 from $F^+_E(c^0)\times \mathbb N \to K$, i.e. functions such that 
 $$S(\mathfrak p):=\{s(\mu):\mathfrak p(\mu,j)\not=0 \ \mbox{for some }  j\in\mathbb N \} \mbox{ is finite}.$$ Any element $\mathfrak p$ in
 $K[[F^+_E(c^0)[\mathbb N ]]]^{sf}$ is  represented by a \emph{source finite}  formal series.
Denote by  $s(\mathfrak p)$ the idempotent
\[s(\mathfrak p):=\sum_{u\in S(\mathfrak p)}u\]
of $L_K(E)$. 
\end{definition}

Again invoking Proposition \ref{prop:U^pdescrizione}, the Pr\"ufer left $L_K(E)$-module $U^E_c$ is also a $K$-subspace of the $K$-vector space
$K[[F^+_E(c^0)[\mathbb N ]]]^{sf}$. Easily,  the set $\{ u\in E^0 \ | \ u \geq v\}$ is finite if and only if   $K[[F^+_E(c^0)[\mathbb N ]]]=K[[F^+_E(c^0)[\mathbb N ]]]^{sf}$.  In particular, this equality holds in case $E$ is a finite graph.  


\begin{remark}\label{injenvelopeisPruferformaxcycleinfinitegraph}
Suppose $c$ is a maximal cycle in a finite graph $E$.   Then $F^+_E(c^0)$ is finite, and so (yet again invoking Proposition \ref{prop:U^pdescrizione}) we have that  $K[[F^+_E(c^0)[\mathbb N ]]]^{sf}$ is precisely $U_c^E$ in this case.  
\end{remark}

\begin{example}
Consider the graph
 \[\xymatrix{ \Gamma:=&u\ar@(ul,dl)[]^e\ar[rr]^{d}&& v\ar@(ul,ur)[]^c\ar[r]^b&w}.\]
 Since 
 $\Gamma$ is finite, we have 
 \[K[[F^+_\Gamma(c^0)[\mathbb N ]]]=K[[F^+_\Gamma(c^0)[\mathbb N ]]]^{sf}.\]
 The elements of $K[[F^+_\Gamma(c^0)[\mathbb N ]]]$ are of the form
 \[\sum_{(\mu,i)\in F^+_\Gamma(c^0)\times \mathbb N }k_{\mu, i} \mu p^*_{v,r(\mu)}\alpha^\Gamma_{c,i}=\sum_{i=1}^{n}k_{i}\alpha^\Gamma_{c,i}+
\sum_{j\geq 0} e^jd\sum_{i=1}^{m_j}h_{j,i}\alpha^\Gamma_{c,i}
 \]
 for suitable $k_{\mu, i}, k_i, h_{j,i}\in K$ and $n,m_j\in\mathbb N $. For example, 
the formal series
 \[\sum_{i\in\mathbb N }e^id\left(\sum_{j\leq i}\alpha^\Gamma_{c,j}\right)\]
  is an element of $K[[F^+_E(c^0)[\mathbb N ]]]$: indeed $e^id\in F^+_E(c^0)$ and for each $i$ the coefficient of $e^id\alpha^\Gamma_{c,j}$ is non zero for finitely many $j\in\mathbb N $, precisely for $1\leq j\leq i$.
\end{example}

\begin{proposition}\label{prop:formalisTmod}
The $K$-vector space $K[[F^+_E(c^0)[\mathbb N ]]]$ is a left $L_K(E)$-premodule, while $K[[F^+_E(c^0)[\mathbb N ]]]^{sf}$ is a left $L_K(E)$-module.
\end{proposition}
\begin{proof}
We claim that there is Cuntz-Krieger $E$-family in ${\rm End}_K(K[[F^+_E(c^0)[\mathbb N ]]])$, which then guarantees the existence of a ring homomorphism $\Phi: L_K(E) \to {\rm End}_K(K[[F^+_E(c^0)[\mathbb N ]]])$, which in turn (by standard ring theory) endows $K[[F^+_E(c^0)[\mathbb N ]]]$ with an $L_K(E)$-premodule structure by setting $$x \cdot m := \Phi(x)(m)$$
for $x\in L_K(E)$ and $m\in K[[F^+_E(c^0)[\mathbb N ]]]$.  We explicitly define the elements $ P_u, S_e, S_{e^*}$ ($u\in E^0, e\in E^1$) of 
${\rm End}_K(K[[F^+_E(c^0)[\mathbb N ]]])$, and demonstrate that this collection forms a  Cuntz-Krieger $E$-family, in the Appendix (Section \ref{sec:appendix}).  
As one can see from their explicit descriptions, the endomorphisms $ P_u, S_e, S_{e^*}$ ($u\in E^0, e\in E^1$) have the effect, respectively, of  multiplying by $u$, $e$, and $e^*$ each of the terms of the formal series $\mathfrak p=\sum_{(\mu,i)\in F^+_E(c^0)\times\mathbb N}k_{\mu,i} \mu p^*_{v,r(\mu)}\alpha^E_{c,i}$.   More formally,
$$P_u(\sum_{(\mu,i)\in F^+_E(c^0)\times\mathbb N}k_{\mu,i} \mu p^*_{v,r(\mu)}\alpha^E_{c,i}) = \sum_{(\mu,i)\in F^+_E(c^0)\times\mathbb N}k_{\mu,i} u\mu p^*_{v,r(\mu)}\alpha^E_{c,i} \ \ \forall \ u \in E^0,$$

$$S_e(\sum_{(\mu,i)\in F^+_E(c^0)\times\mathbb N}k_{\mu,i} \mu p^*_{v,r(\mu)}\alpha^E_{c,i}) = \sum_{(\mu,i)\in F^+_E(c^0)\times\mathbb N}k_{\mu,i} e\mu p^*_{v,r(\mu)}\alpha^E_{c,i} \ \ \forall \ e \in E^1, $$

$$S_{e^*}(\sum_{(\mu,i)\in F^+_E(c^0)\times\mathbb N}k_{\mu,i} \mu p^*_{v,r(\mu)}\alpha^E_{c,i}) = \sum_{(\mu,i)\in F^+_E(c^0)\times\mathbb N}k_{\mu,i} e^*\mu p^*_{v,r(\mu)}\alpha^E_{c,i} \ \ \forall \ e \in E^1.$$


\noindent
Less formally, the vector space of formal series $K[[F^+_E(c^0)[\mathbb N ]]]$ admits the analogous $L_K(E)$-module structure as its $K$-vector subspace $U_c^E$.     

We now consider the $K$-subspace $K[[F^+_E(c^0)[\mathbb N ]]]^{sf}$ of $K[[F^+_E(c^0)[\mathbb N ]]]$.  We claim that $K[[F^+_E(c^0)[\mathbb N ]]]^{sf}$ is closed under the premodule action of $L_K(E)$ on  $K[[F^+_E(c^0)[\mathbb N ]]]$ defined above, and that under this action $K[[F^+_E(c^0)[\mathbb N ]]]^{sf}$ is in fact unitary. 
Because $L_K(E)\cdot L_K(E) = L_K(E)$, to establish the claim it suffices to show that 
 $$L_K(E)\cdot K[[F^+_E(c^0)[\mathbb N ]]]=K[[F^+_E(c^0)[\mathbb N ]]]^{sf}.$$ Let $u\in E^0$ and  $\mathfrak p = \sum_{(\mu,i)\in F^+_E(c^0)\times \mathbb N }k_{\mu,i} \mu p^*_{v,r(\mu)} \alpha^E_{c,i}\in K[[F^+_E(c^0)[\mathbb N ]]]$.   Then using the action defined above we get
 $$ u \cdot \mathfrak p = 
  \sum_{(\mu,i)\in F^+_E(c^0)\times \mathbb N , \ s(\mu) = u}k_{\mu,i} \mu p^*_{v,r(\mu)} \alpha^E_{c,i}.$$
 \noindent 
This yields that  for each $u\in E^0$ and $\mathfrak p \in K[[F^+_E(c^0)[\mathbb N ]]]$ we have 
$$u \cdot \mathfrak p = 0 \ \ \Leftrightarrow \ \ u \notin S(\mathfrak p).$$
 Let $\mathfrak q \in K[[F^+_E(c^0)[\mathbb N ]]]^{sf}$, 
 and write   $\mathfrak q = \sum_{(\nu,i)\in F^+_E(c^0)\times\mathbb N }k_{\nu,i} \nu p^*_{v,r(\nu)}\alpha^E_{c,i}.$ 
  Since $\mathfrak q$ is source finite  the set $S(\mathfrak q):= \{s(\nu) : k_{\nu,i}\neq 0 \ \mbox{for some} \ i\in\mathbb N \}$ is finite; as defined above, we  let $s(\mathfrak q)$ denote the idempotent $\sum_{u\in S(\mathfrak q)}u$ of $L_K(E)$.   Then  it is easy to show that $s(\mathfrak q) \cdot \mathfrak q = \mathfrak q$, which yields \[L_K(E)\cdot K[[F^+_E(c^0)[\mathbb N ]]] \ \supseteq \ L_K(E)\cdot K[[F^+_E(c^0)[\mathbb N ]]]^{sf}\ \supseteq \ K[[F^+_E(c^0)[\mathbb N ]]]^{sf}.\]
 Conversely,  let $r\in L_K(E)$ and $\mathfrak p\in   K[[F^+_E(c^0)[\mathbb N ]]]$.
 Since the set of finite sums of distinct vertices of $E$ forms a set of local units for $L_K(E)$, there exist vertices $u_1, \dots , u_n$ in $E^0$ for which $(\sum_{i=1}^n u_i)  r = r$.   In particular if $u \in E_0 \setminus \{u_1, \dots , u_n\}$ then $ur=0$.   But then for any such $u$ we have $u\cdot (r\cdot \mathfrak p) = (ur) \cdot \mathfrak p = 0$.   Thus by the previous observation $u\notin S(r\cdot \mathfrak p)$, so that $S(r \cdot \mathfrak p) \subseteq \{u_1, \dots , u_n\}$ is finite, and thus $r\cdot \mathfrak p \in K[[F^+_E(c^0)[\mathbb N ]]]^{sf}.$
\end{proof}
\begin{remark}   Indeed it is not hard to show that $K[[F^+_E(c^0)[\mathbb N ]]]$ is an $L_K(E)$-module if and only if $K[[F^+_E(c^0)[\mathbb N ]]] = K[[F^+_E(c^0)[\mathbb N ]]]^{sf}$. 
\end{remark}

  For notational convenience, we denote by $X_c^E$ the left $L_K(E)$-module
$$X_c^E:= K[[F^+_E(c^0)[\mathbb N ]]]^{sf},$$
and we call $X_c^E$ the {\it extended Pr\"{u}fer left $L_K(E)$-module associated to $c$}.

\medskip

We are now in position to establish the Second Piece.  

\begin{lemma}\label{L(E)vessentialinMc}  The simple $L_K(E)$-module $V^E_{c}$ is an essential submodule of the extended Pr\"{u}fer $L_K(E)$-module $X_c^E$.
\end{lemma}
\begin{proof}
We already observed  that $V^E_{c}\subseteq U^E_c\subseteq X_c^E$. Moreover, the left $L_K(E)$-module structure of $X_c^E$ described in 
Proposition~\ref{prop:formalisTmod} makes $V^E_{c}=L_K(E)\alpha^E_{c,1}$ a left $L_K(E)$-submodule of 
$X_c^E$. 
Now, let $\mathfrak p=\sum_{(\mu,i)\in F^+_E(c^0)\times \mathbb N }\mathfrak p(\mu, i)\cdot \mu p^*_{v,r(\mu)}\alpha^E_{c,i}$ be a nonzero element in $X_c^E$. Pick $(\mu_0,i_0)\in F^+_E(c^0)\times \mathbb N $ such that
$\mathfrak p(\mu_0,i_0)\neq 0$. Multiplying $\mathfrak p$ on the left by
$p_{v,r(\mu_0)}\mu_0^*$, by Lemma~\ref{lemma:mu*lambda} and Lemma~\ref{lemma:nuovo c1c1*}(2), we get
\[p_{v,r(\mu_0)}\mu_0^*\mathfrak p = \sum_{i\in\mathbb N }\mathfrak p(\mu_0,i)\alpha^E_{c,i}.
\]
The scalars $\mathfrak p(\mu_0,i)$  are equal to zero for almost all $i\geq 1$ and $\mathfrak p(\mu_0,i_0)\neq 0$. Let $\overline i=\max\{i\in\mathbb N \mid \mathfrak p(\mu_0,i)\neq 0\}$. Then 
\begin{align*}
    \mathfrak p(\mu_0,\overline i)^{-1}(c-v)^{\overline i-1}p_{v,r(\mu_0)}\mu_0^*\cdot \mathfrak p &=
\mathfrak p(\mu_0,\overline i)^{-1}(c-v)^{\overline i-1}\sum_{i\in\mathbb N }\mathfrak p(\mu_0,i)\alpha^E_{c,i}\\
&=\mathfrak p(\mu_0,\overline i)^{-1}(c-v)^{\overline i-1}\sum_{i=1}^{\overline i}\mathfrak p(\mu_0,i)\alpha^E_{c,i}\\
&=\alpha^E_{c,1}\in V^E_{c}.  \qedhere \end{align*}
 \end{proof}

We need one more set of observations prior to establishing our main result.

\begin{lemma}\label{lemma:factors}
Let $J$ be a left ideal of $L_K(E)$ contained in $I(v)$.   Then $J$ is a  left $I(v)$-module.   Furthermore, if $\varphi \in  
 \Hom_{L_K(E)}(J,X_c^E)$, then  $\varphi(J) \subseteq U^E_c$.
\end{lemma}
\begin{proof}
Since $I(v)$ has local units (see \Cref{lem:restrictedversushedgehog}) we have 
\[J=I(v)J,\]
and the first statement follows.  For the same reason,  for each $j\in J$ there exists $x_j\in I(v)$ such that $j=x_j j$. By \cite[Lemma 2.4.1]{AAM} $x_j=\sum_{\ell=1}^m \gamma_\ell\delta^*_\ell$ with $\gamma_\ell,\delta_\ell\in\Path(E)$, $r(\gamma_\ell)=r(\delta_\ell)\in T(v)$. If  $\rho$ denotes the sum of distinct elements in the set $\{r(\gamma_\ell) \ | \ 1\leq \ell \leq m\}$,
then $x_j=\sum_{\ell=1}^m \gamma_\ell \rho \delta^*_\ell$. Thus
\[
\varphi(j)=\varphi(x_j j)=\sum_{\ell=1}^m \gamma_\ell \rho\varphi(\delta_\ell^* j).
\]
If $\varphi(\delta_\ell^* j)=\sum_{(\mu,i)\in F^+_E(c^0)\times\mathbb N }k_{\mu,i} \mu p^*_{v,r(\mu)}\alpha^E_{c,i}
$, since $u\mu=0$ for each $u\in H=T(v)\setminus c^0$, and $\mu\in F^+_E(c^0)$, then
\begin{align*}
    \rho\varphi(\delta_\ell^* j)&=\rho \sum_{(\mu,i)\in F^+_E(c^0)\times\mathbb N }k_{\mu,i} \mu p^*_{v,r(\mu)}\alpha^E_{c,i}\\
    &=\rho\sum_{y=1}^n  v_y \sum_{(\mu,i)\in F^+_E(c^0)\times\mathbb N }k_{\mu,i} \mu p^*_{v,r(\mu)}\alpha^E_{c,i}\\
    &=\rho \sum_{y=1}^n\sum_{i\geq 1}k_{v_y,i}p^*_{v,v_y} \alpha^E_{c,i}\in U^E_c
\end{align*}
because for any $1\leq y\leq n$, $k_{v_y,i}=0$ for almost $i\geq 1$. Thus 
\[
\varphi(j)=\sum_{\ell=1}^m \gamma_\ell \rho\varphi(\delta_\ell^* j)\in U^E_c.
\qedhere \]
\end{proof}

We now have all the pieces in place  to establish the main result of the article.  

\begin{theorem}\label{injhulltheoremc}
Let $E$ be any graph, and let $c$ be an exclusive cycle in $E$.  Then the extended Pr\"{u}fer left $L_K(E)$-module $X_c^E := K[[F^+_E(c^0)[\mathbb N ]]]^{sf}$ is the injective envelope of the simple left $L_K(E)$-module $V^E_{c}$. 
\end{theorem}
\begin{proof}
By The Second Piece (Lemma~\ref{L(E)vessentialinMc}) we know that $V^E_{c}$ is essential in $X_c^E$. 
Thus we need only establish that $X_c^E$ is an injective left $L_K(E)$-module. We use  the aforementioned Proposition~\ref{prop:Baire},  applied here to the two-sided ideal $I(v)$ of $L_K(E)$.     So it suffices to show that:

\smallskip

 (1) if $J$ is a left ideal of $L_K(E)$ contained in $I(v)$ then any $L_K(E)$-homomorphism from $J$ to $X_c^E$ extends to a $L_K(E)$-homomorphism from $I(v)$ to $X_c^E$, and 

(2) if $D$ is any left ideal of $L_K(E)$  containing $I(v)$ then any $L_K(E)$-homomorphism from $D$ to $X_c^E$ extends to a $L_K(E)$-homomorphism from $L_K(E)$ to $X_c^E$.  

\smallskip

To establish (1), we recall that,  by Lemma~\ref{lemma:factors}, if $J$ is a left ideal of $L_K(E)$ contained in $I(v)$ then any $L_K(E)$-homomorphism $\varphi$ from $J$ to $X_c^E$ 
has $\varphi(J) \subseteq U_c^E$.  We rename the homomorphism $\varphi: J \to U_c^E$ as $\hat{\varphi}$.    Since both $J$ and $U^E_c$ are left $I(v)$-modules, and since by the First Piece  (\Cref{cor;Pruferiniettivo}) $U^E_{c}$ is an injective $I(v)$-module, $\hat\varphi$ extends to a homomorphism $\overline\varphi:I(v) \to U^E_{c}$ of left $I(v)$-modules. We prove that $\overline\varphi$ is also a homomorphism of left $L_K(E)$-modules. 
Since \[I(v)=L_K(E)\Sig (T(v)) L_K(E)=I(v) \Sig (T(v)) I(v),\] each $x\in I(v)$ is equal to a suitable sum $\sum_{j=1}^m x_{j}\rho x'_{j}$ with $x_{j}, x'_{j}\in I(v)$, $\rho\in \Sig (T(v))$. Then for each $r\in L_K(E)$ we have
\begin{align*}
    \overline\varphi(rx)&=\overline\varphi(r\sum_{j=1}^m x_{j}\rho x'_{j})=\overline\varphi(\sum_{j=1}^m rx_{j}\rho x'_{j})=\sum_{j=1}^m \overline\varphi((rx_{j})\rho x'_{j})\\
&=\sum_{j=1}^m(rx_{j})\overline\varphi(\rho x'_{j})=\sum_{j=1}^m r(x_{j}\overline\varphi(\rho x'_{j}))=r\sum_{j=1}^m \overline\varphi(x_{j}\rho x'_{j})=r\overline\varphi(x),
\end{align*}
as desired.    Then composing the $L_K(E)$-homomorphism $\overline\varphi$ with the inclusion of $U^E_c$ in $X_c^E$ we have  extended $\varphi:J\to X_c^E$ to a homomorphism $I(v)\to X_c^E$, thus establishing (1).  

We establish (2) as follows.   First we show that the result holds for $D = I(v)$. By Proposition~\ref{lemma:Rmu*}, any $L_K(E)$-homomorphism $\varphi:I(v)\to X_c^E$ is uniquely determined by the values of $\varphi(\mu^*)$, where $\mu \in F^+_E(T(v))$. If $\mu \in F^+_E(T(v))\setminus F^+_E(c^0)$, then $\varphi(\mu^*)=r(\mu)\varphi(\mu^*)=0$: indeed $r(\mu)\in T(v)\setminus c^0$ acts as zero on $X_c^E$.
If $\mu \in F^+_E(c^0)$, let us assume
\[\varphi(\mu^*)=\sum_{(\nu,i)\in F^+_E(c^0)\times \mathbb N }k_{\nu, i,\mu} \nu p^*_{v,r(\nu)}\alpha^E_{c,i}\]
for suitable $k_{\nu, i,\mu}\in K$ .
Since $\varphi(\mu^*)=\varphi(r(\mu)\mu^*)=r(\mu)\varphi(\mu^*)$, and $r(\mu) = r(\mu)^*\in F^+_E(c^0)$, by Lemma~\ref{lemma:mu*lambda} we have 
\begin{align*}
\varphi(\mu^*)&=r(\mu)\varphi(\mu^*)=r(\mu)
\sum_{(\nu,i)\in F^+_E(c^0)\times \mathbb N }k_{\nu, i,\mu} \nu p^*_{v,r(\nu)}\alpha^E_{c,i}\\
&=\sum_{i\in\mathbb N }k_{r(\mu), i,\mu}p_{v,r(\mu)}^*\alpha^E_{c,i}.
\end{align*}

Now, define $\epsilon_\varphi \in K[[F^+_E(c^0)[\mathbb N ]]]$ by setting  
\[\epsilon_\varphi :=\sum_{(\gamma,i)\in F^+_E(c^0)\times \mathbb N }k_{r(\gamma), i,\gamma} \gamma p^*_{v,r(\gamma)}\alpha^E_{c,i}.\]
Since $L_K(E)\cdot K[[F^+_E(c^0)[\mathbb N ]]]=K[[F^+_E(c^0)[\mathbb N ]]]^{sf} := X_c^E$, we can define 
\[\theta: L_K(E) \to X_c^E,\quad \theta(r):=r\cdot \epsilon_\varphi.\]
For each $\mu\in F^+_E(T(v))\setminus F^+_E(c^0)$, we have $\mu^*\cdot \epsilon_\varphi=0$. Therefore, to show that $\theta$ extends $\varphi$, it suffices to show that $\varphi(\mu^*) = \mu^* \cdot \epsilon_\varphi$ for each  $\mu \in F^+_E(c^0)$. But again invoking Lemma~\ref{lemma:mu*lambda} we get 
\begin{align*}
\mu^* \cdot \epsilon_\varphi&=\mu^*\sum_{(\gamma,i)\in F^+_E(c^0)\times \mathbb N }k_{r(\gamma), i,\gamma} \gamma p^*_{v,r(\gamma)}\alpha^E_{c,i}\\
&=
    \sum_{i\in\mathbb N }k_{r(\mu), i,\mu}p^*_{v,r(\mu)}\alpha^E_{c,i},
\end{align*}
which  is precisely $\varphi(\mu^*)$. 

Next, let  $\psi: D\to X_c^E$ be a homomorphism of left $L_K(E)$-modules, where $D$ is a left ideal of $L_K(E)$ containing $I(v)$.   We  prove that the restriction map 
\[\Hom_{L_K(E)}(D, X_c^E)\to \Hom_{L_K(E)}(I(v), X_c^E)\]
is injective, and the conclusion will follow from  Lemma~\ref{lemma:restriction}. On the contrary, assume that $\psi_1,\psi_2: D \to X_c^E$ are two different homomorphisms with $\psi_1(x)=\psi_2(x)$ for all $x\in I(v)$. Since $\psi_1\not=\psi_2$, there exists $y\in D$ such that $\psi_1(y) \not= \psi_2(y)$ in $X_c^E$. Set
\[\psi_1(y)=\sum_{(\mu,i)\in F^+_E(c^0)\times \mathbb N }k_{\mu, i} \mu p^*_{v,r(\mu)}\alpha^E_{c,i}\quad\text{and}\]
\[
\psi_2(y)=\sum_{(\mu,i)\in F^+_E(c^0)\times \mathbb N }h_{\mu, i} \mu p^*_{v,r(\mu)}\alpha^E_{c,i}.\]
So there exists $\gamma \in F^+_E(c^0)$ such that 
\[\sum_{i\in\mathbb N }k_{\gamma,i}p^*_{v,r(\gamma)}\alpha^E_{c,i}\not= \sum_{i\in\mathbb N }h_{\gamma,i}p^*_{v,r(\gamma)}\alpha^E_{c,i}\]
for appropriate $k_{\gamma,i}, h_{\gamma, i} \in K$.  
By Lemma~\ref{lemma:mu*lambda} we have $\psi_1(\gamma^* y)=\gamma^* \cdot\psi_1(y)  = \sum_{i\in\mathbb N }k_{\gamma,i}p^*_{v,r(\gamma)}\alpha^E_{c,i}$, and similarly $\psi_2(\gamma^* y)=\gamma^* \cdot \psi_2(y) =
\sum_{i\in\mathbb N }h_{\gamma,i}p^*_{v,r(\gamma)}\alpha^E_{c,i}$. Since $\gamma^* \in I(v)$ we have  $\gamma^* y \in I(v)$. 
But then 
 \begin{align*}
 \sum_{i\in\mathbb N }k_{\gamma,i}p^*_{v,r(\gamma)}\alpha^E_{c,i} &= \gamma^*  \cdot \psi_1(y) =  \psi_1(\gamma^* y)\\
 &= \psi_2(\gamma^* y)=\gamma^* \cdot \psi_2(y) = \sum_{i\in\mathbb N }h_{\gamma,i}p^*_{v,r(\gamma)}\alpha^E_{c,i},
 \end{align*}
 contrary to the displayed inequality.  

 Thus we have shown that the extended Pr\"{u}fer left $L_K(E)$-module $X_c^E:= K[[F^+_E(c^0)[\mathbb N ]]]^{sf}$ is  injective, which, along with Lemma \ref{L(E)vessentialinMc}, establishes our main result.  
\end{proof}

\begin{remark}
In the First Step above we utilized our result \cite[Theorem 6.4]{AMT19}, which says that in a finite graph $E$ the Pr\"{u}fer $L_K(E)$-module $U_c^E$ is the injective envelope of the simple $L_K(E)$-module $V_c^E$.  Theorem \ref{injhulltheoremc} demonstrates that \cite[Theorem 6.4]{AMT19} does not extend even to graphs with finitely many vertices.  For instance, in the graph
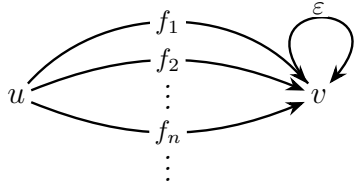
\begin{center}
\begin{tikzpicture}[
  >=Stealth,
  line width=0.9pt,
  vertex/.style={inner sep=1.2pt,outer sep=1pt},
  bigloopabove/.style={
    loop above,
    out=130,
    in=50,
    min distance=15mm,
    looseness=4
  },
  edgelabel/.style={
    midway,
    fill=white,
    inner sep=1.2pt,
    font=\footnotesize
  }
]

\node[vertex] (u) at (0,0) {$u$};
\node[vertex] (v) at (4,0) {$v$};

\draw[->] (v) edge[bigloopabove]
  node[edgelabel] {$\varepsilon$} (v);

\draw[->] (u) to[bend left=48]
  node[edgelabel,pos=0.50,sloped,inner xsep=1.9pt] {$f_1$} (v);

\draw[->] (u) to[bend left=22]
  node[edgelabel,pos=0.50,sloped,inner xsep=1.9pt] {$f_2$} (v);

\node[fill=white,inner sep=-1.1pt] at (2,0.10) {$\vdots$};

\draw[->] (u) to[bend right=22]
  node[edgelabel,pos=0.50,sloped,inner xsep=1.9pt] {$f_n$} (v);

\node[fill=white,inner sep=-1.1pt] at (2,-0.85) {$\vdots$};

\end{tikzpicture}
\end{center}
(there are infinitely many edges from $u$ to $v$), the injective module $X_{\varepsilon}^E$, in which $U_{\varepsilon}^E$ is an essential submodule,  is properly larger than $U_{\varepsilon}^E$.  (For example, the formal series $\sum_{j\in \mathbb{N}} f_j \alpha_{\varepsilon,1}^E$ is an element of $X_\varepsilon^E \setminus U_{\varepsilon}^E$.) In particular, this implies that $U_{\varepsilon}^E$ is not injective as a left $L_K(E)$-module. 

On the other hand, for infinite graphs the Pr\"{u}fer module associated to a maximal cycle may indeed be injective, even if there are infinitely many arrows ending in that maximal cycle. For example, consider the graph $E'$
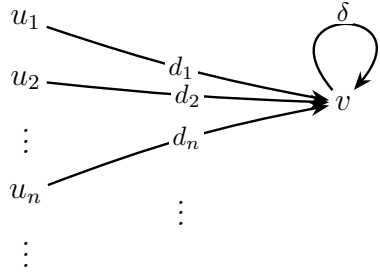
\begin{center}
\begin{tikzpicture}[
  >=Stealth,
  line width=0.9pt,
  vertex/.style={inner sep=1.2pt,outer sep=1pt},
  bigloopabove/.style={
    loop above,
    out=130,
    in=50,
    min distance=15mm,
    looseness=4
  },
  edgelabel/.style={
    midway,
    fill=white,
    inner sep=1.2pt,
    font=\footnotesize
  }
]

\node[vertex] (v) at (4.2,1.25) {$v$};

\node[vertex] (u1) at (0,2.35) {$u_1$};
\node[vertex] (u2) at (0,1.55) {$u_2$};

\node at (0,0.85) {$\vdots$};

\node[vertex] (un) at (0,0.05) {$u_n$};

\node at (0,-0.65) {$\vdots$};

\draw[->] (v) edge[bigloopabove]
  node[edgelabel] {$\delta$} (v);

\draw[->] (u1) to[bend right=4]
  node[edgelabel,pos=0.48] {$d_1$} (v);

\draw[->] (u2) to[bend right=2]
  node[edgelabel,pos=0.50] {$d_2$} (v);

\node[fill=white,inner sep=0.7pt]
  at (2.05,0.88) {$\vdots$};

\draw[->] (un) to[bend left=5]
  node[edgelabel,pos=0.50] {$d_n$} (v);

\node[fill=white,inner sep=0.7pt]
  at (2.05,-0.10) {$\vdots$};

\end{tikzpicture}
\end{center}
Then the extended Pr\"ufer left module $X^{E'}_\delta$ coincides with the Pr\"ufer left module $U^{E'}_\delta$.  We note that $U^{E'}_\delta=X_\delta^{E'} = K[[F_{E'}^+(v)[\mathbb N]]]^{sf}$  is strictly contained in the $L_K(E')$-premodule $K[[F_{E'}^+(v)[\mathbb N]]].$
\end{remark}

\medskip

As mentioned previously, we will use our  main result  as a template for a more general result, in which we present the injective envelope of $V^E_{p(x),c}$, where $p(x)\in K[x,x^{-1}]$ is a basic irreducible polynomial.

For $p(x)\in K[x]$ having ${\rm deg}(p(x))\geq 1$, we denote by $\mathbb Z^+_{<\deg p(x)}$ the set of integers $\{0, 1, 2, \dots, \deg p(x)-1\}.$
Analogously to \Cref{Mcdef,def;KF+0} we give the following 

\begin{definition}
We denote by $K[[F^+_E(c^0)[\mathbb N ][\mathbb Z^+_{<\deg p(x)}]]]$ the $K$-vector space of all functions $\mathfrak p$ from $F^+_E(c^0)\times \mathbb N \times \mathbb Z^+_{<\deg p(x)}$ to $K$ that assume for each $\mu\in F^+_E(c^0)$, $0\leq j<\deg p(x) $, the value $0$ for almost all $i\in\mathbb N $. We represent any element $\mathfrak p\in K[[F^+_E(c^0)[\mathbb N ][\mathbb Z^+_{<\deg p(x)}]]]$ as a ``formal series''
\[\mathfrak p=\sum_{(\mu,i,j)\in F^+_E(c^0)\times \mathbb N \times \mathbb Z^+_{<\deg p(x)}}k_{\mu,i,j} \mu p^*_{s(c), r(\mu)}c^j\alpha_{p(x),c,i}, \]
where $k_{\mu,i,j}=\mathfrak p(\mu,i,j)\in K$, and $\alpha_{p(x),c,i}$ are the generators of the Pr\"ufer left $L_K(E)$-module $U^E_{p(x),c}$.

We denote by $K[[F^+_E(c^0)[\mathbb N ][\mathbb Z^+_{<\deg p(x)}]]]^{sf}$ the $K$-vector subspace of the \emph{source finite} functions belonging to $K[[F^+_E(c^0)[\mathbb N ][\mathbb Z^+_{<\deg p(x)}]]]$, i.e., functions $\mathfrak p$ such that $\{s(\mu): \mathfrak p(\mu,i,j)\not=0\  \exists i\in\mathbb N ,\exists j\in \mathbb Z^+_{<\deg p(x)}\}$ is finite.

For notational convenience we denote by $X_{p(x),c}^E$ the left $L_K(E)$-module
$$X_{p(x),c}^E := K[[F^+_E(c^0)[\mathbb N ][\mathbb Z^+_{<\deg p(x)}]]]^{sf},$$
and we call $X_{p(x),c}^E$ {\rm the extended Pr\"{u}fer left $L_K(E)$-module associated to $p(x)$ and $c$.}
\end{definition}

\begin{remark}
    Observe that each triple $(\mu,i,j)\in F^+_E(c^0)\times \mathbb N \times \mathbb Z^+_{<\deg p(x)}$ uniquely determines the element $\mu p^*_{s(c),r(\mu)}c^j\alpha_{p(x),c,i}\in U^E_{p(x),c}$. Indeed, if
    \[\mu_1 p^*_{s(c),r(\mu_1)}c^{j_1}\alpha_{p(x),c,i_1}=\mu_2 p^*_{s(c),r(\mu_2)}c^{j_2}\alpha_{p(x),c,i_2}\quad i_1\geq i_2\]
 then necessarily $i_1=i_2$: indeed, if $i_1>i_2$, then the first element does not belong to $L_K(E)\alpha_{p(x),c,i_2}$. Then, setting  $i:=i_1=i_2$, we have
 \[\mu_1 p^*_{s(c),r(\mu_1)}c^{j_1}\alpha_{p(x),c,i}=\mu_2 p^*_{s(c),r(\mu_2)}c^{j_2}\alpha_{p(x),c,i}.\]
 Assume now $j_1\leq j_2$. Since $p_{s(c),r(\mu)}p^*_{s(c),r(\mu)}c=p_{s(c),r(\mu)}p_{r(\mu), s(c)}=c$ and 
by \Cref{rem:contialpha-pversion}(7),
multiplying by $(c^{j_1})^*p_{s(c),r(\mu_1)}\mu_1^*$ we get
\[
0\not=\alpha_{p(x),c,i}=(c^{j_1})^*p_{s(c),r(\mu_1)}\mu_1^*\mu_2 p^*_{s(c),r(\mu_2)}c^{j_2}\alpha_{p(x),c,i}\quad\forall 0\leq j_1\leq j_2.\]
Necessarily, we get $\mu_1=\mu_2=:\mu$ by \Cref{lemma:mu*lambda} and then
\begin{align*}
    0\not=\alpha_{p(x),c,i}&=(c^{j_1})^*p_{s(c),r(\mu)}p^*_{s(c),r(\mu)}c^{j_2}\alpha_{p(x),c,i}\\
    &=(c^{j_1})^*c^{j_2}\alpha_{p(x),c,i}=c^{j_2-j_1}\alpha_{p(x),c,i}\quad\Rightarrow \quad j_1=j_2.
\end{align*}
 Therefore, it is not inappropriate to use $\mu p^*_{s(c),r(\mu)}c^j\alpha_{p(x),c,i}$ as a placeholder for the scalar $k_{\mu,i,j}$. This notation allows us to naturally view the Pr\"ufer left $L_K(E)$-module $U^E_{p(x),c}$ as a $K$-subspace of the $K$-vector space $K[[F^+_E(c^0)[\mathbb N ][\mathbb Z^+_{<\deg p(x)}]]]$ (see \Cref{prop:U^pdescrizione}).
\end{remark}

We present the generalized version of our main result here.  Its proof can be achieved using an approach completely  analogous to the one used above to establish Theorem \ref{injhulltheoremc}, along with the germane ideas presented in Section \ref{section:LPA}.  
\begin{theorem}\label{generalinjhulltheoremc}   Let $E$ be any graph, and $K$ any field. Let $c$ be an exclusive cycle in $E$, and $p(x)$ a basic irreducible polynomial in $K[x,x^{-1}]$.  Then the extended Pr\"{u}fer left $L_K(E)$-module associated to $p(x)$ and $c$, i.e., 
\[ X_{p(x),c}^E := K[[F^+_E(c^0)[\mathbb N ][\mathbb Z^+_{<\deg p(x)}]]]^{sf},\] is the injective envelope of the simple left $L_K(E)$-module $V^E_{p(x),c}$. 
\end{theorem}

\begin{example}
Consider the following graph $E$:
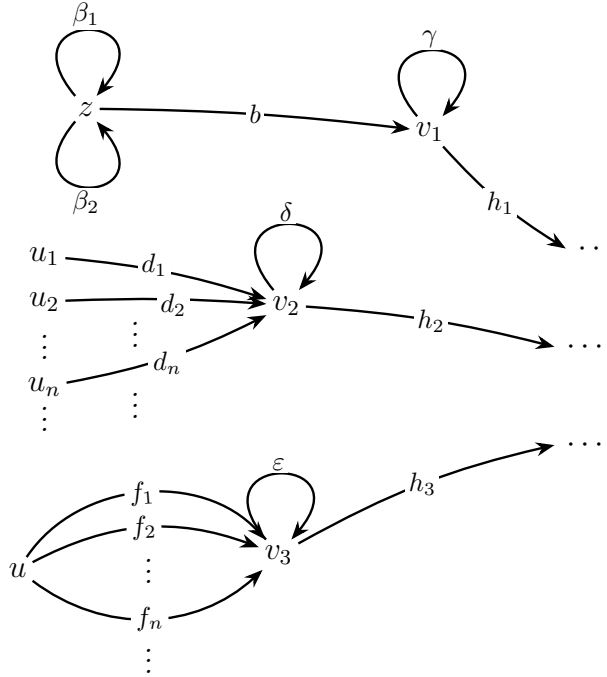
    \begin{center}

    \begin{tikzpicture}[
  >=Stealth,
  line width=0.9pt,
  vertex/.style={inner sep=1.2pt,outer sep=1pt},
  bigloopabove/.style={loop above,out=130,in=50,min distance=15mm,looseness=4},
  bigloopbelow/.style={loop below,out=230,in=310,min distance=15mm,looseness=4},
  edgelabel/.style={midway,fill=white,inner sep=1.2pt,font=\footnotesize}
]

\node[vertex] (z) at (0,6.5) {$z$};

\node[vertex] (v1) at (4.55,6.20) {$v_1$};
\node[vertex] (v2) at (2.65,3.90) {$v_2$};
\node[vertex] (v3) at (2.55,0.60) {$v_3$};

\node[vertex] (u1) at (-0.55,4.55) {$u_1$};
\node[vertex] (u2) at (-0.55,3.95) {$u_2$};

\node at (-0.55,3.45) {$\vdots$};
\node at (0.65,3.62) {$\vdots$};

\node[vertex] (un) at (-0.55,2.82) {$u_n$};

\node at (-0.55,2.50) {$\vdots$};
\node at (0.65,2.68) {$\vdots$};

\node[vertex] (u) at (-0.90,0.40) {$u$};

\draw[->] (z)  edge[bigloopabove] node[edgelabel] {$\beta_1$} (z);
\draw[->] (z)  edge[bigloopbelow] node[edgelabel] {$\beta_2$} (z);
\draw[->] (v1) edge[bigloopabove] node[edgelabel] {$\gamma$} (v1);
\draw[->] (v2) edge[bigloopabove] node[edgelabel] {$\delta$} (v2);
\draw[->] (v3) edge[bigloopabove] node[edgelabel] {$\varepsilon$} (v3);

\draw[->] (z) to[bend left=3]
  node[edgelabel] {$b$} (v1);

\draw[->] (u1) to[bend left=6]
  node[edgelabel,pos=0.43,yshift=2pt] {$d_1$} (v2);

\draw[->] (u2) to[bend left=3]
  node[edgelabel,pos=0.54,yshift=-2pt] {$d_2$} (v2);

\draw[->] (un) to[bend right=8]
  node[edgelabel,pos=0.49,yshift=-2pt,inner xsep=1.8pt] {$d_n$} (v2);

\draw[->] (u) to[bend left=52]
  node[edgelabel,pos=0.50,sloped,inner xsep=1.9pt] {$f_1$} (v3);

\draw[->] (u) to[bend left=24]
  node[edgelabel,pos=0.50,sloped,inner xsep=1.9pt] {$f_2$} (v3);

\draw[->] (u) to[bend right=40]
  node[edgelabel,pos=0.50,sloped,inner xsep=1.9pt] {$f_n$} (v3);

\node[fill=white,inner sep=0.7pt] at (0.825,0.49) {$\vdots$};
\node[fill=white,inner sep=0.7pt] at (0.825,-0.70) {$\vdots$};

\coordinate (r1) at (6.35,4.65);
\coordinate (r2) at (6.20,3.35);
\coordinate (r3) at (6.20,2.05);

\draw[->] (v1) to[bend right=8]
  node[edgelabel,pos=0.50] {$h_1$} (r1);

\draw[->] (v2) to[bend left=5]
  node[edgelabel,pos=0.50] {$h_2$} (r2);

\draw[->] (v3) to[bend left=8]
  node[edgelabel,pos=0.50] {$h_3$} (r3);

\node[right] at (r1) {$\cdots$};
\node[right] at (r2) {$\cdots$};
\node[right] at (r3) {$\cdots$};

\end{tikzpicture}
\end{center}
The loops $\gamma$, $\delta$, and $\varepsilon$, are exclusive cycles.
Denoting by $\{\beta_1,\beta_2\}^*$ the set of all words (the empty word included) in the alphabet $\{\beta_1,\beta_2\}$, we have
\begin{align*}
    F_E^+(v_1)=&\big\{v_1,\pi(\beta_1,\beta_2)b: \pi(\beta_1,\beta_2)\in\{\beta_1,\beta_2\}^*\big\}\\
    F_E^+(v_2)=&\big\{v_2, d_i: i\geq 1\big\}\\
    F_E^+(v_3)=&\big\{v_3, f_i:i\geq 1\big\}
\end{align*}
Therefore by \Cref{injhulltheoremc} the injective envelopes of the simple modules $V^E_\gamma$, $V^E_\delta$, and $V^E_\varepsilon$, are, respectively, 
\begin{align*}
    X^E_\gamma&=K[[F_E^+(v_1)[\mathbb N]]]^{sf}=K[[F_E^+(v_1)[\mathbb N]]]\\
    &=\big\{
    \sum_{i=1}^n k_i \alpha^E_{\gamma,i}+\sum_{\pi(\beta_1,\beta_2)\in\{\beta_1,\beta_2\}^*} k_{\pi(\beta_1,\beta_2),i}\pi(\beta_1,\beta_2)b\alpha^E_{\gamma,i}\mid k_i, k_{\pi(\beta_1,\beta_2),i}\in K,\\
    & \forall \pi(\beta_1,\beta_2)\in\{\beta_1,\beta_2\}^*: k_{\pi(\beta_1,\beta_2),i}=0  \text{ for almost every }i\in\mathbb N\big\},\\
 \end{align*}
 \begin{align*}
    X^E_\delta&=K[[F_E^+(v_2)[\mathbb N]]]^{sf} \\
    &=\big\{\sum_{i=1}^nk_i \alpha^E_{\delta,i}+\sum_{j=1}^m\sum_{i=1}^n k_{j,i}d_j\alpha^E_{\delta,i}\mid k_i, k_{j,i}\in K, m,n\in\mathbb N\big\} =U^E_\delta, \ \mbox{and} 
    \end{align*}
    \begin{align*}
    X^E_\varepsilon&=K[[F_E^+(v_3)[\mathbb N]]]^{sf}=K[[F_E^+(v_3)[\mathbb N]]]\\
    &=\big\{\sum_{i=1}^nk_i \alpha^E_{\varepsilon,i}+\sum_{i,j\geq 1}k_{i,j}f_j\alpha^E_{\varepsilon,i}\mid k_i, k_{i,j}\in K,\\
    &\forall j\geq 1: k_{i,j}=0 \text{ for almost every }i\in\mathbb N\big\}.
\end{align*}    

Assume now that $K=\mathbb Q$ and  fix the basic irreducible polynomial $p(x)=-x^3-x-1\in\mathbb Q[x,x^{-1}]$. Then by \Cref{generalinjhulltheoremc} the injective envelopes of the simple modules $V^E_{p(x),\gamma}$, $V^E_{p(x),\delta}$, and $V^E_{p(x),\varepsilon}$, are, respectively, 
\begin{align*}
    X^E_{p(x),\gamma}&=\mathbb Q[[F_E^+(v_1)[\mathbb N][\mathbb Z^+_{<3}]]]^{sf}=\mathbb Q[[F_E^+(v_1)[\mathbb N][\mathbb Z^+_{<3}]]]\\
    &=\big\{
    \sum_{i=1}^n \sum_{\ell=0}^2 k_{i,\ell} \gamma^\ell\alpha^E_{p(x),\gamma,i}+\sum_{\pi(\beta_1,\beta_2)\in\{\beta_1,\beta_2\}^*} \sum_{\ell=0}^2k_{\pi(\beta_1,\beta_2),i,\ell}\pi(\beta_1,\beta_2)b\gamma^\ell\alpha^E_{p(x),\gamma,i}\mid \\
    &k_{i,\ell}, k_{\pi(\beta_1,\beta_2),i,\ell}\in \mathbb Q,\text{ and }
     \forall \pi(\beta_1,\beta_2)\in\{\beta_1,\beta_2\}^*: \\
     &k_{\pi(\beta_1,\beta_2),i,\ell}=0  \text{ for almost every }i\in\mathbb N\big\},
 \end{align*}
 \begin{align*}
    X^E_{p(x),\delta}&=\mathbb Q[[F_E^+(v_2)[\mathbb N][\mathbb Z^+_{<3}]]]^{sf}\\
    &=\big\{\sum_{i=1}^n\sum_{\ell=0}^2k_{i,\ell} \delta^\ell\alpha^E_{p(x),\delta,i}+\sum_{j=1}^m\sum_{i=1}^n \sum_{\ell=0}^2 k_{j,i,\ell}d_j\delta^\ell\alpha^E_{p(x),\delta,i}\mid\\
    &k_{i,\ell}, k_{j,i,\ell}\in \mathbb Q, n\in\mathbb N\big\}, \ \mbox{and} \ 
    \end{align*}
    \begin{align*}
    X^E_{p(x),\varepsilon}&=\mathbb Q[[F_E^+(v_3)[\mathbb N][\mathbb Z^+_{<3}]]]^{sf}=\mathbb Q[[F_E^+(v_3)[\mathbb N][\mathbb Z^+_{<3}]]]\\
    &=\big\{\sum_{i=1}^n\sum_{\ell=0}^2k_{i,\ell} \varepsilon^\ell\alpha^E_{p(x),\varepsilon,i}+\sum_{i,j\geq 1}\sum_{\ell=0}^2k_{i,j,\ell}f_j\varepsilon^\ell\alpha^E_{p(x),\varepsilon,i}\mid k_{i,\ell}, k_{i,j, \ell}\in \mathbb Q,\\
    &\text{such that }\forall j\geq 1: k_{i,j, \ell}=0 \text{ for almost every }i\in\mathbb N\big\}.
\end{align*}  

\end{example}

\section{Appendix}\label{sec:appendix}

In the proof of Proposition \ref{prop:formalisTmod} we asserted that there is a left $L_K(E)$-premodule structure on $K[[F^+_E(c^0)[\mathbb N ]]]$ when $c$ is an exclusive cycle.   In this Appendix we formally define that premodule action. To do so, we need only to find a collection of $K$-linear endomorphisms of  $K[[F^+_E(c^0)[\mathbb N ]]]$ that form a Cuntz-Krieger $E$-family, and then invoke the Universal Property of $L_K(E)$ (see e.g., \cite[Remark 1.2.5]{AAM}).

 To that end, let $u\in E^0$, $e\in E^1$. We define three $K$-linear endomorphisms $P_u$, $S_e$, and $S_{e^*}$ of the $K$-vector space $K[[F^+_E(c^0)[\mathbb N ]]]$ of all functions $\mathfrak p:F^+_E(c^0)\times\mathbb N \to K$ that assume for each $\mu\in F^+_E(c^0)$ the value 0 for almost all $i\in\mathbb N $. For each  $(\mu,i)\in F^+_E(c^0)\times\mathbb N $ we define
\[
P_u(\mathfrak p)(\mu,i):=\begin{cases}
    \mathfrak p(\mu,i)&\text{if $s(\mu)=u$};\\
    0&\text{otherwise.}
\end{cases}
\]

\[
S_e(\mathfrak p)(\mu,i):=\begin{cases}\begin{cases}
    \mathfrak p(\mu',i)  & \text{if }\mu=e\mu', \\
     0 & \text{otherwise}
\end{cases} & \text{if }s(e)\not\in c^0;\\
\begin{cases}
    0&\text{if }r(e)\notin c^0\\
    \mathfrak p(r(e),i)  & \text{if } \mu=s(e), v\not=r(e)\in c^0, 
    \\
 \mathfrak p(r(e),i+1)+\mathfrak p(r(e),i)   & \text{if } \mu=s(e), r(e)=v\\
     0 & \text{otherwise}
\end{cases}
 & \text{if }s(e)\in c^0.
\end{cases}
\]

\[
S_{e^*}(\mathfrak p)(\mu,i):=\begin{cases}\begin{cases}
    \mathfrak p(e\mu,i)  & \text{if }s(\mu)=r(e), \\
     0 & \text{otherwise}
\end{cases} & \text{if }s(e)\not\in c^0;\\
\begin{cases}
    \mathfrak p(s(e),i)  & \text{if } \mu=r(e)\not=v \\
 \displaystyle\sum_{j\geq i}(-1)^{j-i}\mathfrak p(s(e),j)   & \text{if } \mu=r(e)=v\\
     0 & \text{otherwise}
\end{cases}
 & \text{if }s(e)\in c^0.
\end{cases}
\]
Observe that $\displaystyle \sum_{j\geq i}(-1)^{j-i}\mathfrak p(s(e),j)$ is the sum of finitely many nonzero scalars, since for each $\mu\in F^+_E(c^0)$ we have by definition that $\mathfrak p(\mu,j)=0$ for $j\gg 0$.

Clearly $P_u$, $S_e$, and $S_{e^*}$ are $K$-linear. If
\[\mathfrak p=\sum_{(\mu,i)\in F^+_E(c^0)\times\mathbb N }k_{\mu,i} \mu p^*_{v,r(\mu)}\alpha^E_{c,i}\]
then


\begin{align*}
P_u(\mathfrak p)&=
\sum_{\substack{(\mu,i)\in F^+_E(c^0)\times\mathbb N \\ s(\mu)=u}}
k_{\mu,i} \mu p^*_{v,r(\mu)}\alpha^E_{c,i}=\sum_{\substack{(\mu,i)\in F^+_E(c^0)\times\mathbb N}}
k_{\mu,i}u\mu p^*_{v,r(\mu)}\alpha^E_{c,i}\\
S_e(\mathfrak p)&=\begin{cases}
\displaystyle\sum_{\substack{
(\mu,i)\in F^+_E(c^0)\times\mathbb N
\\
\mu=e\mu'}
}k_{\mu',i} \mu p^*_{v,r(\mu)} \alpha^E_{c,i} =
   \sum_{
   (\mu,i)\in F^+_E(c^0)\times\mathbb N
}
k_{\mu,i} e\mu p^*_{v,r(\mu)}\alpha^E_{c,i}   & \text{if }s(e)\notin c^0,\\
0 &\text{if }s(e)\in c^0, r(e)\notin c^0\\
\displaystyle\sum_{i\geq 1}k_{r(e),i}p^*_{v,s(e)}\alpha^E_{c,i}= 
\displaystyle\sum_{
(\mu,i)\in F^+_E(c^0)\times\mathbb N
}
k_{\mu,i} e_j\mu p^*_{v,r(\mu)}\alpha^E_{c,i}
   & \text{if }e=e_j, 1\leq j<n\\
\displaystyle\sum_{i\geq 1}(k_{v,i+1}+k_{v,i})p^*_{v,s(e)}\alpha^E_{c,i}=
\sum_{(\mu,i)\in F^+_E(c^0)\times\mathbb N }k_{\mu,i} e_n\mu p^*_{v,r(\mu)}\alpha^E_{c,i}
& \text{if }e=e_n.
\end{cases}
\end{align*}
The latter equality, for $e=e_n$, follows by 
\begin{align*}
\sum_{i\geq 1}(k_{v,i+1}+k_{v,i})p^*_{v,s(e)}\alpha^E_{c,i}&=\sum_{i\geq 1}k_{v,i}p^*_{v,v_n}(\alpha^E_{c,i}+\alpha^E_{c,i-1})\\
&=
\sum_{i\geq 1}k_{v,i}p^*_{v,v_n}c\alpha^E_{c,i}=
\sum_{i\geq 1}k_{v,i}e_n\alpha^E_{c,i}\\
&=
\sum_{(\mu,i)\in F^+_E(c^0)\times\mathbb N }k_{\mu,i} e_n\mu p^*_{v,r(\mu)}\alpha^E_{c,i} \ .
\end{align*}

\begin{align*}
S_{e^*}(\mathfrak p)&=\begin{cases}
  \displaystyle\sum_{\substack{
(\mu,i)\in F^+_E(c^0)\times\mathbb N\\
s(\mu)=r(e)}
}k_{e\mu,i} \mu p^*_{v,r(\mu)} \alpha^E_{c,i} = \sum_{
(\mu,i)\in F^+_E(c^0)\times\mathbb N}
k_{\mu,i} e^*\mu p^*_{v,r(\mu)}\alpha^E_{c,i} & \text{if }s(e)\notin c^0, \\
0 & \text{if }s(e)\in c^0, r(e)\notin c^0
\\
\displaystyle\sum_{i\geq 1}k_{v_j,i}p^*_{v,v_{j+1}}\alpha^E_{c,i}= 
\displaystyle\sum_{
(\mu,i)\in F^+_E(c^0)\times\mathbb N}
k_{\mu,i} e_j^*\mu p^*_{v,r(\mu)}\alpha^E_{c,i}
   & \text{if }e=e_j, 1\leq j<n\\
\displaystyle\sum_{i\geq 1}
\sum_{j\geq i}(-1)^{j-i}k_{v,j}\alpha^E_{c,i}=
\sum_{(\mu,i)\in F^+_E(c^0)\times\mathbb N }k_{\mu,i} e^*_n\mu p^*_{v,r(\mu)}\alpha^E_{c,i}
& \text{if }e=e_n.
\end{cases}
\end{align*}
The latter equality, for $e=e_n$, follows by 
\begin{align*}
    \sum_{(\mu,i)\in F^+_E(c^0)\times\mathbb N }k_{\mu,i} e^*_n\mu p^*_{v,r(\mu)}\alpha^E_{c,i}&=\sum_{i\geq 1}k_{v,i}c^*\alpha^E_{c,i}\\
    &=\sum_{i\geq 1}k_{v,i}\sum_{j\leq i}(-1)^{i-j}\alpha^E_{c,j}\\
    &=\sum_{i\geq 1}\left(\sum_{j\geq i}(-1)^{j-i} k_{v,j}\right)\alpha^E_{c,i} \ .
\end{align*}
Hence, the effect corresponds to multiplying by $u$, $e$ and $e^*$ all the summands of the formal series.  
 
 \textbf{Claim}: The subset $\{P_u, S_e, S_{e^*}\mid u\in E^0, e\in E^1\}$ forms a Cuntz-Krieger $E$-family in $\End_K(K[[F^+_E(c^0)[\mathbb N ]]])$.
 
\textbf{Proof of the Claim}:  Formally there are five separate types of equations we must verify.   The verifications are cumbersome but straightforward. To gain a sense of the computations involved, we suggest the reader consult \cite[Appendix]{Paper1}, in which we present such computations in the context of building the injective envelope of a simple left $L_K(E)$-ideal.  

We content ourselves here by establishing just one of the many required types of equations.  Specifically, we demonstrate that, for each regular vertex $u\in E^0$,
\[
\sum_{e\in s^{-1}(u)}S_eS_{e^*}=P_u.
\]
If $u\notin c^0$, we have
\begin{align*}
   \sum_{e\in s^{-1}(u)}S_e(S_{e^*}(\mathfrak p))(\mu,i)&=S_{e^*}(\mathfrak p) (\mu',i)\quad \text{if \ $\exists \ e\in s^{-1}(u)$ with $\mu=e \mu'$}\\
   &=\mathfrak p (\mu,i)\quad \text{if $s(\mu)=u$}\\
   &=P_u (\mathfrak p) (\mu,i).
\end{align*}
If $u=v_j$, $1\leq j<n$, we have
\begin{align*}
   \sum_{e\in s^{-1}(u)}S_e(S_{e^*}(\mathfrak p)) (\mu,i)&=S_{e_j^*}(\mathfrak p) (v_{j+1},i)\\
   &=\mathfrak p (v_j,i)= P_{v_j} (\mathfrak p) (\mu,i).
\end{align*}
If $u=v_n$,
we have
\begin{align*}
   \sum_{e\in s^{-1}(u)}S_e(S_{e^*}(\mathfrak p)) (\mu,i)&=S_{e_n^*}(\mathfrak p) (v,i+1)+S_{e_n^*}(\mathfrak p) (v,i)\\
   &=\sum_{j\geq i+1}(-1)^{j-i-1}\mathfrak p(v_n,j)+\sum_{j\geq i}(-1)^{j-i}\mathfrak p(v_n,j)\\
   &=\mathfrak p(v_n,i)= P_{v_n} (\mathfrak p) (\mu,i) . 
\end{align*}

\end{document}